\documentclass[
  a4paper, 
  reqno, 
  oneside, 
  12pt
]{amsart}

\usepackage[utf8]{inputenc}

\usepackage[dvipsnames]{xcolor} % Must come before tikz!
\colorlet{cite}{red}
\usepackage{tikz}  % for diagrams
\usepackage{float}

\usetikzlibrary{arrows,positioning}
\tikzset{ 
  baseline=-2.3pt,
  text height=1.5ex, text depth=0.25ex,
  >=stealth,
  node distance=2cm,
  mid/.style={fill=white,inner sep=2.5pt},
}

\usepackage{lmodern}
\usepackage{tikz-cd}
\usepackage{amsthm, amssymb, amsfonts}

\usepackage{amsthm, amssymb, amsfonts}	% Standard maths packages.
\usepackage[all]{xy}
\usepackage{graphicx,caption,subcaption}
\usepackage{braket}  % For nice sets
\usepackage{microtype}
\usepackage{DotArrow}
\usepackage[makeroom]{cancel}
\usepackage[%
  bookmarks=true,			% show bookmarks bar?
  unicode=true,			% non-Latin characters in Acrobat’s bookmarks
  pdftitle={Topology SF basic 1-forms},		%
  pdfauthor={LopezGarcia-Valencia},	% 
  pdfkeywords={Topology SF, basic 1-forms, groupoids},	% list of keywords
  colorlinks=true,		% false: boxed links; true: colored links
  linkcolor=black,			% color of internal links
  citecolor=black,		% color of links to bibliography
  filecolor=magenta,		% color of file links
  urlcolor=RoyalBlue			% color of external links
]{hyperref}				% One of the most useful packages I have found.
\usepackage{cleveref}

\newtheoremstyle{mydef}
  {}		% Space above environment
  {}		% Space below environment
  {}		% Body font
  {}		% Indent amount (empty = no indent, \parindent = para indent)
  {\scshape}	% theorem head font
  {. }		% Punctuation after heading
  { }		% Space after heading
  {\thmname{#1}\thmnumber{ #2}\thmnote{ #3}}	% Heading spec

\newtheorem{theorem}{Theorem}[section]
\newtheorem*{theorem*}{Theorem}
\newtheorem{proposition}[theorem]{Proposition}
\newtheorem*{proposition*}{Proposition}
\newtheorem{lemma}[theorem]{Lemma}
\newtheorem*{lemma*}{Lemma}
\newtheorem{corollary}[theorem]{Corollary}
\newtheorem*{corollary*}{Corollary}
\theoremstyle{definition}
\newtheorem{definition}[theorem]{Definition}
\newtheorem{example}[theorem]{Example}

\theoremstyle{remark}
\newtheorem{remark}[theorem]{Remark}

\newtheorem*{conjecture*}{Conjecture}
\usepackage{tikz}

\newcommand{\rr}{\rightrightarrows}
\newcommand{\F}{\mathcal{F}}

\author{Daniel L\'opez-Garcia {\tiny and} Fabricio Valencia}
\address{}
\date{\today}

\dedicatory{Dedicated to Amaru}

\address{D. L\'opez-Garcia - Instituto de Matem\'atica e Estat\'istica, Universidade Federal Fluminense,  Rua Prof. Marcos Waldemar de Freitas Reis s/n, 24210-201 Niter\'oi - Brazil.  \newline  
      \phantom{xx}
       F. Valencia - Instituto de Matem\'atica e Estat\'istica, Universidade de S\~ao Paulo, Rua do Mat\~ao 1010, Cidade Universit\'aria, 05508-090 S\~ao Paulo - Brazil.
      \newline  
      \phantom{xx}
  danielflg@id.uff.br, fabricio.valencia@ime.usp.br}

\title{Topology of singular foliations of closed 1-forms on orbifolds}

\begin{document}
\maketitle

\begin{abstract}
We study the topological properties of the leaves of the singular foliation induced by a closed 1-form of Morse type on a compact orbifold. In particular, we establish criteria that characterize when all such leaves are compact, when they are non-compact, and how both types may coexist. As an application, we extend to the orbifold setting a celebrated result of Calabi, which provides a purely topological characterization of intrinsically closed harmonic 1-forms of Morse type.
\end{abstract}

\tableofcontents
\section{Introduction}

Classical Morse theory is a powerful tool that enables one to extract both geometric and topological information from a manifold by studying a smooth real-valued function whose critical points are all nondegenerate. In the early 1980s, Novikov initiated a generalization of Morse theory by considering the zeros of closed 1-forms instead of the critical points of functions, an approach that led to numerous applications in geometry, topology, analysis, and dynamical systems (see, e.g. \cite{F} and its quoted references). Novikov's initial ideas were based on constructing a chain complex using the dynamics of the gradient flow lifted to the abelian cover associated with the cohomology class of a given closed 1-form \cite{No1, No2}. This approach provided a natural extension of the classical Morse inequalities, relying on the key observation that, in Morse theory, the dynamics of the gradient flow traditionally serve as a bridge between the critical set of a function and the global topology of the underlying space. Two decades later, Calabi raised the problem of whether it is possible to improve the Novikov inequalities for closed 1-forms with Morse-type zeros if one additionally assumes that the 1-form is harmonic with respect to a Riemannian metric \cite{Calabi}, thus introducing the notion of transitivity. These kinds of questions motivated Farber--Katz--Levine \cite{FKL} and Honda \cite{Honda} to explore transitive closed 1-forms, leading to a detailed study of the geometric properties of the singular foliations defined by closed 1-forms with Morse-type zeros.

This paper aims to extend the main results of both Calabi and Farber--Katz--Levine to the setting of compact orbifolds, which are among the simplest singular spaces that still support smooth-type structures. Morse theory on orbifolds, as well as on other types of singular spaces, has been widely studied, and the existing literature on the subject is extensive (see, e.g., \cite{ChoHong,Corrigan,GM,Hep,HolmMatsu,LT,OV}). Nevertheless, a full extension of Novikov theory to the setting of orbifolds, capable of handling closed 1-forms with Morse-type zeros, remains largely undeveloped, with the exception of the recent work \cite{V}, where preliminary results in this direction have been established and Novikov-type inequalities for compact orbifolds have been proven.

To state our main results, we first briefly introduce some terminology, following primarily the conventions and definitions in \cite{V}. We use proper and étale Lie groupoids as atlases for orbifolds, allowing us to view an orbifold as the orbit space of a Lie groupoid of this type. Such a  perspective proves particularly useful when working with various geometric and algebraic structures on orbifolds, or more generally, on singular orbit spaces. Let $X$ be a compact orbifold presented by a proper and étale Lie groupoid $G\rr M$. We think of the orbifold fundamental group $\Pi_1^\textnormal{orb}(X,[x])$ as the set consisting of $G$-homotopy classes of $G$-loops at $x\in M$ endowed with the group product induced by concatenation of $G$-paths \cite{BH,MoMr}. If $\overline{\omega}$ is a closed 1-form on $X$ presented by a closed basic 1-form $\omega$ on $M$, then we define the \emph{singular foliation} $\tilde{\mathcal{F}}_\omega$ of $X$ as the decomposition of $X$ into leaves defined as follows. Two points $[x]$ and $[y]$ in $X$ are said to be related, or to lie on the same leaf, if and only if there exists a smooth $G$-path $\sigma=\sigma_ng_n\sigma_{n-1}\cdots \sigma_1g_1\sigma_0$ in $M$ such that $\sigma_0(0)=x$, $\sigma_n(1)=y$ and $\omega(\dot{\sigma}_k(\tau))=0$ for all $k=0,1,\cdots,n$ and $\tau\in [0,1]$. The leaves of $\tilde{\mathcal{F}}_\omega$ containing the zeros of $\overline{\omega}$ are called singular. We say that $\overline{\omega}$ is generic if each singular leaf of $\tilde{\mathcal{F}}_\omega$ contains precisely one zero of $\overline{\omega}$. If $\overline{\mathcal{L}}$ is a singular leaf then its singular leaf components are defined to be the closures in $\overline{\mathcal{L}}$ of the connected components of $\overline{\mathcal{L}}\backslash  \lbrace [x]: \overline{\omega}([x])=0\rbrace$.

In these terms, our first main result characterizes when all leaves of the singular foliation $\tilde{\mathcal{F}}_\omega$ are compact and explains how compact and non-compact leaves coexist in $X$. This can be stated as follows.

\begin{theorem}\label{thm1}
Let $X$ be a compact orbifold of dimension $n$ and let $\overline{\omega}$ denote a closed 1-form of Morse type on $X$ representing a cohomology class $\xi\in H^1(X)$.
\begin{enumerate}
\item All leaves of the singular foliation $\tilde{\mathcal{F}}_\omega$ are
compact if and only if the homomorphism of periods $\Pi_1^{\textnormal{orb}}(X)\to (\mathbb{R},+)$ induced by $\xi$ can be factorized 
as 
$$\Pi_1^{\textnormal{orb}}(X)\to F\to (\mathbb{R},+),$$
through a free group $F$. 
\item $X$ is the union of two compact $n$-dimensional suborbifolds $X_c$ and $X_\infty$ with a common, possibly singular, boundary satisfying:
\begin{enumerate}
	\item $\textnormal{int}(X_c)$ is the union of all of the compact leaves (nonsingular and singular) plus some compact singular leaf components of non-compact singular leaves,
	\item $\textnormal{int}(X_\infty)$ is the union of all noncompact nonsingular leaves, all noncompact singular leaf components and some compact singular leaf components of noncompact singular leaves,
	\item $\partial  X_c=\partial X_\infty=X_c\cap X_\infty$ is the union of finitely many compact singular leaf components of noncompact singular leaves. It is an orbifold except at finitely many points which are zeros of $\overline{\omega}$,
	\item if $X_\infty\neq \emptyset$ then the rank of the cohomology class $\xi_\infty=\xi|_{X_\infty}\in H^1(X_\infty)$ is greater than $1$, and
	\item if $\overline{\omega}$ is generic then the boundary $\partial  X_c=\partial X_\infty=X_c\cap X_\infty$ is the union of all compact singular leaf components of non-compact singular leaves. It is a closed $(n-1)$-dimensional topological orbifold.
\end{enumerate}
\end{enumerate}
\end{theorem}

A closed 1-form $\overline{\omega}$ on $X$ is called intrinsically harmonic if it is harmonic with respect to some Riemannian metric on $X$. Equivalently, the closed basic 1-form $\omega$ on $M$ presenting $\overline{\omega}$ is harmonic with respect to some groupoid Riemannian metric on $G\rr M$. A smooth $G$-path $\sigma=\sigma_ng_n\sigma_{n-1}\cdots \sigma_1g_1\sigma_0$ in $M$ is said to be $\omega$-positive if $\omega(\dot{\sigma}_k(\tau))>0$ for all $k=0,1,\cdots,n$ and $\tau\in [0,1]$. Geometrically, this means that the velocity vector $\dot{\sigma}_k(\tau)$ of each path $\sigma_k(\tau)$ always points in the direction in which $\overline{\omega}$ increases. Accordingly, a closed 1-form $\overline{\omega}$ on $X$ is called \emph{transitive} if for any point $[x]\in X\backslash \textnormal{zeros}(\overline{\omega})$ there exists an $\omega$-positive smooth $G$-loop at $x\in M$. 

Using the topological description for the leaves of the singular foliation $\tilde{\mathcal{F}}_\omega$ provided in Theorem \ref{thm1}, we can also show a purely topological characterization of intrinsically closed harmonic 1-forms of Morse type on compact orbifolds in terms of transitivity. This is the content of our second main result.

\begin{theorem}\label{thm3}
	Let $X$ be a compact orbifold. A closed 1-form of Morse type on $X$ is intrinsically harmonic if and only if it is transitive. 
\end{theorem}

It is worth stressing that, building upon the works \cite{FKL,Honda}, several other interesting results concerning the Morse-theoretic properties of harmonic closed 1-forms on compact orbifolds can be established by employing similar techniques (see for instance Corollaries \ref{cor1} and \ref{cor2} and Proposition \ref{pro1}). Such results may be used to determine whether one can improve the topological lower bounds for the number of zeros of a closed 1-form of Morse type on a compact orbifold. As alluded to previously, the first attempt in that direction is provided by the Novikov type inequalities, which were recently established in \cite{V}.

The paper is organized as follows. In Section \ref{sec:2}, we provide a brief overview of the main notions and terminology used throughout the paper, closely following \cite{V}. In particular, we define closed 1-forms of Morse type, introduce the notions of $G$-paths and the orbifold fundamental group, and construct the homomorphism of periods associated with closed 1-forms on orbifolds. Section \ref{sec:3} is devoted to study some topological properties of a natural foliation-type object related to any closed 1-form of Morse type on a compact orbifold, which we refer to as the singular foliation. After showing some preliminary results, we characterize when all leaves of such singular foliations are compact and give a complete description of how compact and non-compact leaves co-occur within the ambient orbifold. This is the content of Theorem \ref{thm1}. In Section \ref{sec:4}, we deal with intrinsically harmonic closed 1-forms on Riemannian orbifolds and define the notion of transitivity in terms of positive $G$-paths. We prove several interesting results around the notion of transitive closed 1-form of Morse type, thus employing some of them to characterize harmonic closed 1-forms of Morse type by means of transitivity, see Theorem \ref{thm3}. Finally, in Section \ref{sec:5}, we provide some examples which allow to illustrate our main results. Most of these examples are obtained by adapting the connected sum operation constructions developed in \cite{FKL}.

\vspace{.2cm}
{\bf Acknowledgments:}  We are grateful to Fernando Studzinski for several insightful discussions. D. L\'opez-Garcia was supported by Grant 2022/04705-8 Sao Paulo Research Foundation - FAPESP. F. Valencia was supported by Grant 2024/14883-6 Sao Paulo Research Foundation - FAPESP.

\section{Preliminaries}\label{sec:2}

In this short section we briefly introduce the main notions and terminology that we will be using throughout the paper. We assume that the reader is familiar with the notion of Lie groupoid and the geometric/topological aspects underlying its structure \cite{dH,MoMr0}. Let $G\rr M$ be a Lie groupoid. The structural maps of $G$ are denoted by $(s,t,m,u,i)$ where $s,t:G\to M$ are the maps respectively indicating the source and target of the arrows, $m:G^{(2)}\to G$ stands for the partial composition of arrows, $u:M\to G$ is the unit map which sometime we shall denote by $u(x):=1_x$ for all $x\in M$, and $i:G\to G$ is the map determined by the 
inversion of arrows. Here we are denoting by $G^{(2)}=\lbrace (g,h): s(g)=t(h)\rbrace\subseteq G\times G$ the manifold of composable arrows. The orbit space $M/G$ of $G\rr M$ is denoted by $X$ and the canonical orbit projection by $\pi:M\to X$. If $\mathcal{O}_x$ denotes the groupoid orbit through $x\in M$ then we shall often consider its normal bundle $\nu(\mathcal{O}_x)$ as the bundle whose fiber at $x$ is given by $\nu_x(\mathcal{O}_x)=T_xM/T_x\mathcal{O}_x$.

We say that a Lie groupoid $G\rr M$ is a \emph{proper} if the source/target map $(s,t):G\to M\times M$ is proper. In this case, the groupoid orbits $\mathcal{O}_x$ are embedded in $M$, the isotropy groups $G_x$ are compact, and the orbit space $X$ is Hausdorff, second-countable, and paracompact \cite{dH}. Moreover, the orbit projection $\pi:M\to X$ is an open map. The Lie groupoid is said to be \emph{étale} if either $s$ or $t$ is a local diffeomorphism, meaning that $G$ and $M$ have the same dimension. In particular, the isotropies $G_x$ and the orbits $\mathcal{O}_x$ are discrete, so that $\nu_x(\mathcal{O}_x)=T_xM$. We use proper étale groupoids as orbifold atlases, so that we can think of an orbifold as being the orbit space of a Lie groupoid of this kind \cite{BX,Ler}. Such a point of view turns out to be quite useful when dealing with several geometric and algebraic notions concerning the structure of an orbifold, or even a singular orbit space in general.

From now on we assume that $G\rr M$ is proper étale Lie groupoid presenting a compact orbifold $X$. It is well-known that differential forms on $X$ are completely determined by basic forms on $M$, compare \cite{PPT,TuX,Wa}. That is, differential forms $\omega$ on $M$ such that $(t^\ast-s^\ast)(\omega)=0$. The space of basic differential forms on $M$ is denoted by $\Omega_{\textnormal{bas}}^\bullet(G)$ and it is identified with the space of differential forms on $X$ that is denoted by $\Omega^\bullet(X)$. Note that the de Rham differential over $\Omega^\bullet(M)$ restricts to $\Omega_{\textnormal{bas}}^\bullet(G)$, thus yielding the so-called \emph{basic cohomology} $H_{\textnormal{bas}}^\bullet(G,\mathbb{R})$ of $G$. This cohomology is Morita invariant, so that it recovers the de Rham cohomology $H^\bullet(X)$ of $X$. Additionally, there is an isomorphism $H^\bullet(X,\mathbb{R})\cong H_{\textnormal{bas}}^\bullet(G,\mathbb{R})$, where $H^\bullet(X,\mathbb{R})$ stands for the singular cohomology of $X$. It is worth mentioning that $X$ can be triangulated so that its cohomology $H^\bullet(X)\approx H_{\textnormal{bas}}^\bullet(G,\mathbb{R})$ becomes a finite dimensional vector space \cite{PPT}.

\subsection{Closed 1-forms of Morse type}

Morse functions on $X$ are completely determined by basic functions on $M$ whose critical orbits are nondegenerate in the sense of Morse--Bott theory \cite{Bo,Hep,OV}. Motivated by the approach presented in \cite{OV} to address  Morse theory for differentiable stacks, the notion of stacky closed 1-form of Morse type was recently introduced in \cite{V}. Let $\omega$ be a closed basic $1$-form on $M$. The set of zeros of $\omega$ is saturated in $M$, so that it is formed by a disjoint union of groupoid orbits. Therefore, we say that an orbit $\mathcal{O}_x$ of zeros of $\omega$ is \emph{nondegenerate} if its normal Hessian is a nondegenerate fiberwise bilinear symmetric form on $\nu(O_x)$. Accordingly, $\omega$ is said to be of \emph{Morse type} if all of its orbits of zeros are nondegenerate. In order to be precise, one knows that for each groupoid orbit $\mathcal{O}_x$ there exist an open neighborhood $\mathcal{O}_x\subset U\subset M$ and a basic smooth function $f_U\in \Omega_{\textnormal{bas}}^0(G|_{U})$ such that $\omega|_U=df_U$, see \cite[Lem. 8.5]{PPT}. If $U$ is connected then the function $f_U$ is determined by $\omega|_U$ uniquely up to a constant. In particular, $\textnormal{zeros}(\omega|_{U})=\textnormal{Crit}(f_U)$, so that the \emph{normal Hessian} of $\omega$ along an orbit of zeros $\mathcal{O}_x$ is defined to be the the normal Hessian of $f_U$ along the critical orbit $\mathcal{O}_x$. Hence, an orbit of zeros $\mathcal{O}_x$ of $\omega$ is \emph{nondegenerate} if and only if $f_U$ is nondegenerate along $\mathcal{O}_x$.

Closed basic 1-forms of Morse type are Morita invariant (see \cite[Prop. 3.8]{V}), so they give rise to the following definition.

\begin{definition}
A \emph{closed 1-form of Morse type} on $X$ is defined to be the Morita equivalence class $\overline{\omega}$ of a closed basic $1$-form of Morse type $\omega$ on $M$.
\end{definition}

We think of $\overline{\omega}$ as the assignment $X\ni [x]\mapsto \overline{\omega}([x])=\omega(x)$ for each $ x\in M$, which is clearly well-defined. In consequence, a zero $[x]$ of $\overline{\omega}$ corresponds to an orbit of zeros $\mathcal{O}_x$ of $\omega$, so that it is nondegenerate if and only if its corresponding orbit is nondegenerate. 

Let us fix a Riemannian metric on $X$, which can be thought of as an equivalence class of a groupoid Riemannian metric on $G\rr M$ in the sense of del Hoyo and Fernandes in \cite{dHF,dHF2}. Since the normal Hessian $\textnormal{Hess}(f_U)$ is nondegenerate it follows that the normal bundle $\nu(\mathcal{O}_x)$ splits into the Whitney sum of two subbundles $\nu_-(\mathcal{O}_x)\oplus \nu_+(\mathcal{O}_x)$ such that $\textnormal{Hess}(f_U)$ is strictly negative on $\nu_-(\mathcal{O}_x)$ and strictly positive on $\nu_+(\mathcal{O}_x)$. From \cite[Lem. 5.4]{OV} we know that $\textnormal{Hess}(f_U)$ is invariant with respect to the normal representation $G_x\curvearrowright \nu(\mathcal{O}_x)_x$, so that it preserves the splitting above since the normal representation is by isometries in this case. In consequence, we get a normal sub-representation $G_x\curvearrowright \nu_-(\mathcal{O}_x)_x$, inducing actions of $G_x$ over the unit disk $D_-(\mathcal{O}_x)_x$ and its boundary $\partial D_-(\mathcal{O}_x)_x$, both of which sit clearly inside $\nu_-(\mathcal{O}_x)_x$. The \emph{index data} of $\overline{\omega}$ at the zero $[x]\in X$ is defined as the pair $\lambda_{[x]}=(\lambda_x,G_x)$, where $\lambda_x=\textnormal{rk}(\nu_-(\mathcal{O}_x))$ is called the \emph{index} of $[x]$. Note that the basic function $f_U:U\to \mathbb{R}$ induces a function $F_U:[U/G_U]\subset X\to \mathbb{R}$. Hence, if $[a,b]\subset \mathbb{R}$ is a real interval such that $F_U^{-1}[a,b]$ contains no critical points besides $[x]$ then $\lbrace F_U\leq b\rbrace$ has the homotopy type of $\lbrace F_U\leq a\rbrace$ with a copy of $D_-(\mathcal{O}_x)_x/G_x$ attached along $\partial D_-(\mathcal{O}_x)_x/G_x$, see \cite{Hep,OV}. 

\subsection{Homomorphism of periods}

Before defining the homomorphism of periods of any closed 1-form on $X$ we need to introduce the the notion of $G$-path in $M$ which is a crucial concept in this paper. The reader is recommended to consult \cite[s. 3.3]{MoMr} and \cite[c. G; s. 3]{BH} for specific details. As alluded to previously, we think of paths in $X$ as $G$-paths in $M$. Namely, a \emph{$G$-path} in $M$ is a sequence $\sigma:=\sigma_ng_n\sigma_{n-1}\cdots \sigma_1g_1\sigma_0$ where $\sigma_0,\cdots,\sigma_n:[0,1]\to M$ are (piecewise smooth) paths in $M$ and $g_1,\cdots,g_n$ are arrows in $G$ such that $g_j:\sigma_{j-1}(1)\to \sigma_{j}(0)$ for all $j=1,\cdots,n$. One says that $\sigma$ is a $G$-path of \emph{order} $n$ from $\sigma_{0}(0)$ to $\sigma_{n}(1)$. Our groupoid $G$ is said to be \emph{$G$-connected} if for any two points $x,y\in M$ there exists a $G$-path from $x$ to $y$. We always assume that the groupoids we are working with are $G$-connected unless otherwise stated. If $\sigma':=\sigma_n'g_n'\sigma_{n-1}'\cdots \sigma_1'g_1'\sigma_0'$ is another $G$-path in $M$ with $\sigma_{0}'(0)=\sigma_{n}(1)$ then we can \emph{concatenate} $\sigma$ and $\sigma'$ into a new $G$-path 
$$\sigma\ast \sigma'=\sigma_n'g_n'\sigma_{n-1}'\cdots \sigma_1'g_1'\sigma_0'1_{\sigma_{n}(1)}\sigma_ng_n\sigma_{n-1}\cdots \sigma_1g_1\sigma_0.$$

We define an equivalence relation in the set of $G$-paths in $M$ which is generated by the following \emph{multiplication equivalence}
$$\sigma_n g_n\cdots \sigma_{j+1}g_{j+1}\sigma_jg_j\sigma_{j-1}\cdots g_1\sigma_0 \quad\sim\quad \sigma_n g_n\cdots \sigma_{j+1}g_{j+1}g_j\sigma_{j-1}\cdots g_1\sigma_0,$$
if $\sigma_j$ is the constant path for any $0<j<n$, and  \emph{concatenation equivalence}
$$\sigma_n g_n\cdots g_{j+1}\sigma_jg_j\sigma_{j-1}g_{j-1}\cdots g_1\sigma_0 \quad\sim\quad \sigma_n g_n\cdots g_{j+1}\sigma_j\cdot\sigma_{j-1}g_{j-1}\cdots g_1\sigma_0,$$
if $g_j=1_{\alpha_{j-1}(1)}$ for any $0<j<n$ where $\sigma_j\cdot\sigma_{j-1}$ stands for the standard concatenation of the paths $\sigma_j$ and $\sigma_{j-1}$. A \emph{deformation} between two $G$-paths $\sigma$ and $\sigma'$ of the same order $n$ from $x$ to $y$ consists of homotopies $D_j:[0,1]\times [0,1]\to M$ from $\sigma_j$ to $\sigma_j'$ for $j=0,1,\cdots, n$ and paths $d_i:[0,1]\to G$ from $g_j$ to $g_j'$ for $j=1,\cdots, n$ such that $s\circ d_j=D_{j-1}(\cdot,1)$ and $t\circ d_j=D_{j}(\cdot,0)$ for $j=1,\cdots, n$ verifying $D_0([0,1],0)=x$ and $D_n([0,1],1)=y$. That is, a deformation is a continuous family of $G$-paths of order $n$ from $x$ to $y$ which may be written as $D_n(\tau,\cdot)d_n(\tau)\cdots d_1(\tau)D_0(\tau,\cdot)$ for $\tau\in [0,1]$. Accordingly, two $G$-paths with fixed endpoints in $M$ are said to be \emph{$G$-homotopic} if it is possible to pass from one to another by a sequence of equivalences and deformations. With the multiplication induced by concatenation it follows that the $G$-homotopy classes of $G$-paths form a Lie groupoid over $M$ which is called the \emph{fundamental groupoid} of $G$ and is denoted by $\Pi_1(G)\rr M$. The \emph{fundamental group} of $G$ with respect to a base-point $x_0\in M$ is the isotropy group $\Pi_1(G, x_0):= \Pi_1(G)_{x_0}$. Note that it consists of $G$-homotopy classes of $G$-loops at $x_0$ which are by definition the $G$-homotopy classes of $G$-paths from $x_0$ to $x_0$. It is simple to check that for any two different points $x_0,y_0\in M$ it holds that $\Pi_1(G, x_0)$ and $\Pi_1(G, y_0)$ are isomorphic by $G$-connectedness. Every Lie groupoid morphism $\phi_1:G\to G'$ covering $\phi_0:M\to M'$ induces another Lie groupoid morphism $(\phi_1)_\ast:\Pi_1(G)\to \Pi_1(G')$ by mapping $[\sigma]$ to $[(\phi_1)_\ast(\sigma)]$, where $(\phi_1)_\ast(\sigma)=(\phi_0)_\ast(\sigma_n)\phi_1(g_n)(\phi_0)_\ast(\sigma_{n-1})\cdots (\phi_0)_\ast(\sigma_1)\phi_1(g_0)(\phi_0)_\ast(\sigma_0)$. This also covers $\phi_0:M\to M'$ so that it induces a Lie group homomorphism between the fundamental groups  $(\phi_1)_\ast:\Pi_1(G,x_0)\to \Pi_1(G',\phi_0(x_0))$. As an important feature, we have that Morita equivalent groupoids have isomorphic fundamental groupoids, so that we define the \emph{orbifold fundamental group} of $X$ at $[x_0]$ as $\Pi_1^\textnormal{orb}(X,[x_0])=\Pi_1(G,x_0)$.

\begin{remark}\label{remark1}
There are two interesting scenarios in which we can give a description of the orbifold fundamental group in much more familiar terms. First, let $(M,\mathcal{F})$ be a proper foliation and let $X\approx M/\mathcal{F}$ denote the orbifold presented by the holonomy groupoid $\textnormal{Hol}(M,\mathcal{F})\rr M$. If all principal leaves are 1-connected then one has an isomorphism
$$\Pi_1^\textnormal{orb}(X,[x_0])=\Pi_1(M,x_0),$$
where $\Pi_1(M,x_0)$ stands for the standard fundamental group of $M$ at $x_0$, see \cite[Lem. 6.1]{CFMT}. Second, let us suppose that $K$ is a discrete group acting properly and effectively on a smooth manifold $M$, so that the action groupoid $K\ltimes M\rr M$ represents an orbifold $X\approx M/K$. The set of $(K\ltimes M)$-loops based at $x\in M$ is in one-to-one correspondence with the set of pairs $(\sigma,k)$ where $\sigma:[0,1]\to M$ is a smooth path starting at $x$ and $k\in K$ is such that $k\sigma(0)=\sigma(1)$, compare \cite[p. 608]{BH} and \cite[p. 192]{MoMr}. Additionally, there exists a short exact sequence of groups:
\begin{equation*}
	1\to \Pi_1(M,x_0)\to \Pi_1^\textnormal{orb}(X,[x_0])\to K\to 1.
\end{equation*}
\end{remark}

\begin{remark}\label{remark2}
We can also explicitly describe the space of differential forms on the types of orbifolds introduced in Remark \ref{remark1}. Indeed, differential forms on $M/\mathcal{F}$ are completely determined by differential forms on $M$ which are \emph{basic} and \emph{horizontal} with respect to the foliation $\mathcal{F}$. That is, differential forms $\alpha\in \Omega^\bullet (M)$ satisfying $\iota_v\alpha=0$ and $L_v\alpha=0$ for all $v$ tangent to the leaves. Similarly, differential forms on $M/K$ are completely determined by differential forms on $M$ which are $K$-invariant and horizontal, the latter term meaning that it vanishes on vectors tangent to the $K$-orbits. As shown in \cite{Wa}, a differential form is basic and horizontal in the previous senses if and only if it is a basic differential form on the corresponding holonomy and actions groupoids.
\end{remark}

Let $\overline{\omega}$ be a closed 1-form on $X$ presented by a closed basic 1-form $\omega$ on $M$. For each smooth $G$-path $\sigma=\sigma_ng_n\sigma_{n-1}\cdots \sigma_1g_1\sigma_0$ in $M$ we define the $G$-\emph{path integral}
$$\int_{\sigma}\overline{\omega}=\sum_{k=0}^{n}\int_{\sigma_k}\omega.$$

\noindent If $[\sigma]$ denotes the $G$-homotopy class of $\sigma$ then the expression $\int_{[\sigma]}\overline{\omega}=\int_{\sigma}\overline{\omega}$ is well-defined. More importantly, suppose that $\omega$ and $\omega'$ are cohomologous closed basic 1-forms. That is, there is a basic smooth function $f:M\to \mathbb{R}$ such that $\omega-\omega'=df$. Hence, it is simple to check that $\int_\sigma (\omega-\omega')=f(\sigma_n(1))-f(\sigma_0(0))$, implying that the expression $\int_{[\sigma]}\xi=\int_{\sigma}\overline{\omega}$ is also well-defined for the cohomology class $\xi=[\omega]\in H^\bullet(X)$, provided that $\sigma$ is a $G$-loop. Additionally, we get a well-defined group homomorphism $\textnormal{Per}_{\xi}:\Pi_1^\textnormal{orb}(X,[x_0])\to (\mathbb{R},+)$ by sending $[\sigma]\mapsto \int_{\sigma}\overline{\omega}$, which is completely determined by $\xi$. We refer to $\textnormal{Per}_{\xi}$ as the \emph{homomorphism of periods} associated with the cohomology class $\xi$. See \cite{V} for details.

\section{The singular foliation}\label{sec:3}

This section is devoted to the study of some topological properties of a natural foliation-type object associated with any closed 1-form of Morse type on a compact orbifold. Let $X$ be a compact orbifold presented by a proper étale Lie groupoid $G\rr M$ and let $\overline{\omega}$ denote a closed 1-form of Morse type on $X$ presented by a closed basic 1-form $\omega$ of Morse type on $M$. The \emph{basic foliation} of $M$ with respect to $\omega$ is by definition the decomposition of $M$ into leaves determined by the following equivalence relation: two points $x,y\in M$ belong to the same ``\emph{leaf}'' if and only if there exists a smooth $G$-path $\sigma=\sigma_ng_n\sigma_{n-1}\cdots \sigma_1g_1\sigma_0$ in $M$ such that $\sigma_0(0)=x$, $\sigma_n(1)=y$ and $\omega(\dot{\sigma}_k(\tau))=0$ for all $k=0,1,\cdots,n$ and $\tau\in [0,1]$. We denote such a basic foliation by $\mathcal{F}_\omega$. From \cite{OV,V} it follows that there are finitely many leaves of $\mathcal{F}_\omega$ containing the orbits of zeros of $\omega$. We refer to those leaves containing the orbits of zeros of $\omega$ as being \emph{singular}.
 
Let us assume for a moment that $\omega=df$ with $f:M\to \mathbb{R}$ basic. 
 
 \begin{lemma}\label{Lemma1}
 Suppose that there exists a smooth $G$-path $\sigma=\sigma_ng_n\sigma_{n-1}\cdots \sigma_1g_1\sigma_0$ in $M$ such that $\sigma_0(0)=x$, $\sigma_n(1)=y$, and $df(\dot{\sigma}_k(\tau))=0$ for all $k=0,1,\cdots,n$ and $\tau\in [0,1]$. Then, $x$ and $y$ sit inside the same level set of $f$.
 \end{lemma}
 \begin{proof}
 This directly follows from the result proved in \cite[s. 3.1]{V}. In fact, since $f$ is basic we have that
 	$$f(\sigma_n(1))-f(\sigma_0(0))=\int_{\sigma}df=\sum_{k=0}^n\int_{0}^1df_{\sigma_k(\tau)}(\dot{\sigma}_k(\tau))=0.$$
 
 \end{proof}
 
In general, we know that for each nondegenerate orbit of zeros $\mathcal{O}_x$ of $\omega$ there exist\footnote{Recall that this is actually true for every groupoid orbit.} a connected open neighborhood $\mathcal{O}_x\subset U$ and a basic smooth function $f_U:U\to \mathbb{R}$ determined up to constant such that $\omega|_U=df_U$. Therefore, by Lemma \ref{Lemma1}, the foliation $\mathcal{F}_\omega$ in $U$ is determined by the level sets of $f_U$. By choosing a covering $\lbrace U\rbrace$ of $M$ we might think that all these foliations match together to form the foliation $\mathcal{F}_\omega$ of $M$. 
 
As alluded to previously, the foliation $\mathcal{F}_\omega$ has finitely many singular leaves which contain the nondegenerate orbits of zeros of $\omega$. One can easily describe their structure locally around each nondegenerate orbit $\mathcal{O}_x$. In fact, by the groupoid Morse lemma (see \cite[Thm. 5.2]{OV}) and by shrinking $U$ if necessary, there exists a groupoid tubular neighborhood (i.e. a linearization) $\phi: (\nu(G_{\mathcal{O}_x})_V\rr V)\xrightarrow[]{\cong} (G_U\rightrightarrows U)$ such that $\phi^\ast F_U=Q_{F_U}$, where $F_U:G_U\to \mathbb{R}$ is given by $s^\ast f_U=t^\ast f_U$. Here $Q_{F_U}:\nu(G_{\mathcal{O}_x})\to \mathbb{R}$ stands for the basic function which is quadratic along the normal fibers and defined by $Q_{F_U}(v)=\frac{1}{2}\textnormal{Hess}_{\pi(v)}(F_U)(v,v)$
where $\pi:\nu(G_{\mathcal{O}_x})\to G_{\mathcal{O}_x}$ is the bundle projection. Therefore, the latter implies that the structure of $\mathcal{F}_\omega$ around $\mathcal{O}_x$ has the form $Q_{f_U}(v)=\textnormal{constant}$.

We are interested in expressing the notions mentioned above as corresponding notions for the orbifold $X$. Let $\mathcal{L}$ be a leaf of $\mathcal{F}_\omega$. On the one hand, from the very definition it is simple to check that if $x$ belongs to $\mathcal{L}$ then the whole groupoid orbit $\mathcal{O}_x$ through $x$ is contained in $\mathcal{L}$. In particular, the leaves of $\mathcal{F}_\omega$ are saturated in $M$, as they are formed by unions of groupoid orbits, and we can define topological subgroupoids $G_\mathcal{L}\rr \mathcal{L}$ of $G\rr M$ which are honest Lie subgroupoids provided $\mathcal{L}$ does not contain orbits of zeros of $\omega$. By pushing the leaves $\mathcal{L}$ of $\mathcal{F}_\omega$ down by means of the orbit projection $\pi:M\to X$ we obtain a partition $\tilde{\mathcal{F}}_\omega$ of $X$ by subspaces $\overline{\mathcal{L}}$ which we refer to as the \emph{singular foliation} in $X$ induced by $\overline{\omega}$. Note that $[x],[y]\in \overline{\mathcal{L}}$ if and only if $x,y\in \mathcal{L}$. On the other hand, two points $x,y\in M$ lie on different leaves if for any smooth $G$-path $\sigma$ from $x$ to $y$ it follows that the $G$-path integral $\int_{\sigma}\overline{\omega}\neq 0$. Similarly for points in $X$.

\begin{remark}\label{RemakrStructure}
There is also a way to locally describe the structure of the leaves of $\tilde{\mathcal{F}}_\omega$ around any point in $X$. The basic function $f_U:U\to \mathbb{R}$ induces a function $[U/G_U]\subset X\to \mathbb{R}$, so that $\overline{\omega}|_{U/G_U}=\overline{df_U}$. If $U/G_U$ does not contain zeros of $\overline{\omega}$ (i.e. $U$ does not contain orbits of zeros of $\omega$) then the structure of the leaves of $\tilde{\mathcal{F}}_\omega$ around $U/G_U$ is given by fibers (level sets) $\overline{f}_U^{-1}(c)$ which correspond to the orbifold $[f_U^{-1}(c)/F_U^{-1}(c)]$ presented by the proper étale Lie groupoid $F_U^{-1}(c)\rr f_U^{-1}(c)$. As expected, $\dim [f_U^{-1}(c)/F_U^{-1}(c)]=\dim X-1$. Otherwise, if $U/G_U$ contains a zero $[x]$ of $\overline{\omega}$ then the structure of $\tilde{\mathcal{F}}_\omega$ around $[x]$ has the form $\overline{Q}_{F_U}([v])=\textnormal{constant}$, where $\overline{Q}_{F_U}$ is the quadratic form associated to the Hessian $\overline{\textnormal{Hess}}_{[x]}(F_U)$ on $T_{[x]}X\approx \nu_x(\mathcal{O}_x)/G_x$. The latter is completely determined by the normal Hessian $\textnormal{Hess}_x(f_U)$ along $\mathcal{O}_x$, since it is invariant by the action of the normal isotropy representation $G_x\curvearrowright \nu_x(\mathcal{O}_x)$. More concretely, let us consider the splitting $\nu_-(\mathcal{O}_x)\oplus \nu_+(\mathcal{O}_x)$ where $\textnormal{Hess}(f_U)$ is strictly negative on $\nu_-(\mathcal{O}_x)$ and strictly positive on $\nu_+(\mathcal{O}_x)$. Recall that such a splitting has the property of being invariant under the normal action of $G_x$. Hence, using the Morse lemma once again \cite{Hep,OV}, one may think of $\overline{\omega}$, in coordinates around $[x]$, as being described by the formula
$$\omega=-\sum_{j=1}^{\lambda_x} y_jdy_j+\sum_{j=\lambda_x+1}^{\dim X} y_jdy_j,$$
with the expression on the right hand side being invariant by isotropies. See \cite{OV} for details.
\end{remark}

It is simple to check that given a compact regular leaf $\overline{ \mathcal{L}}$ there exists a neighborhood $\overline{\mathcal{L}}\subset W\subset X$ so that the singular foliation $\tilde{\mathcal{F}}_\omega$ in $W$ is given by the level sets of a function $W\to \mathbb{R}$. This fact can be thought of as a sort of local Reeb stability theorem for orbifolds (see \cite[s. 2.3]{MoMr0}), and it can be similarly proven by using the tubular neighborhood theorem for orbifolds, compare \cite[Ch. 4]{Choi}. Indeed:
	\begin{lemma}\label{Lemma:localReebStability} Let $\overline{\mathcal{L}}$ be a compact leaf of $\tilde \F_\omega$ with no zeros of $\overline{\omega}$. Then, there exists a neighborhood $W$ of $\overline{\mathcal{L}}$ and a function $\overline{f}_W:W\to \mathbb{R}$ such that $\tilde{\mathcal{F}}_{\omega}$ in $W$ is defined by the level sets of $\overline{f}_W$.
	\end{lemma}
	\begin{proof}
	First of all, one has that the pairing $H^1(X)\times H_1(X)\to \mathbb{R}$ sending $(\xi\times [\sigma])\mapsto \textnormal{Per}_{\xi}(h[\sigma])$ is nondegenerate\footnote{This follows from the well-definition of $\textnormal{Per}_\xi$ and the standard de Rham theorem for orbifolds.}. Here $h:\Pi_1^\textnormal{orb}(X)\to H_1(X)$ stands for the orbifold Hurewicz homomorphism, consult \cite[Prop. 2.2]{V}. Since $\overline{\mathcal{L}}$ is compact, then it  is a suborbifold. Thus, the exponential map defines an isomorphism $\overline{\mathcal{L}}\times I\cong W$, where $W$ is a neighborhood of $\overline{\mathcal{L}}$ in $X$. The fact that $W$ is homotopy equivalent to $\overline{\mathcal{L}}$ implies that $H_1(W)\simeq H_1(\overline{\mathcal{L}})$ and we have a nondegenerate pairing $H^1(W)\times H_1(\overline{\mathcal{L}})\to \mathbb{R}$, so that $\textnormal{Per}_{\xi|W}=0$, where $\xi=[\omega]\in H^1(X)$. That is to say, $\overline{\omega}$ is zero in $H^1(W)$. 
	\end{proof}

The following notion will be crucial later on.

\begin{definition}
A closed 1-form of Morse type $\overline{\omega}$ on $X$ is called \emph{generic} if each singular leaf of $\tilde{\mathcal{F}}_\omega$ contains precisely one zero of $\overline{\omega}$.
\end{definition}

It is clear that the condition of $\overline{\omega}$ being generic can be phrased in terms of $\omega$.

\begin{lemma}\label{lem:generic}
If $\overline{\omega}$ is a closed basic $1$-form of Morse type on $X$ then there exists a small perturbation $\overline{\omega}'$ which is generic, lies in the same cohomology class $[\omega]=[\omega']=\xi$, and satisfies $\overline{\omega}|_{U}=\overline{\omega'}|_{U}$ where $U\subset X$ is an open subset containing all zeros of $\overline{\omega}$.
\end{lemma}
\begin{proof}
Recall that we have a finite amount of nondegenerate zeros $[x_1],\cdots,[x_k]$ of $\overline{\omega}$ in $X$, as it is compact. Since $G\rr M$ is proper and étale, for each $j=1,\cdots,k$ we can construct a smooth basic function $f_j:M\to \mathbb{R}$ such that $f_j|_{\pi^{-1}(U_j)}\equiv 1$ for an open neighborhood $U_j$ of $[x_j]$ in $X$ and $\textnormal{supp}(f_j)\cap \textnormal{supp}(f_i)=\emptyset$ for $i\neq j$, see \cite[s. 3.1]{CM}. Here $\pi:M\to X$ denotes the canonical orbit projection. Let $\sigma_{ij}$ be a smooth $G$-path in $M$ from $x_j$ to $x_i$ and set
$$\omega'=\omega+\sum_{j=1}^{k}a_jdf_j,$$
where all the $a_j\in \mathbb{R}$ are small enough and such that for any $i\ne j$ the number $\int_{\sigma_{ij}}\omega +a_i-a_j$ does not belong to the subgroup of periods $\textnormal{Per}_\xi(\Pi_1^\textnormal{orb}(X)) \subset \mathbb{R}$, compare \cite[s. 3.1]{V}. It is clear that $\omega'$ is closed, basic, and belongs to the same basic cohomology class of $\omega$. Furthermore, the critical orbits of $\omega'$ are given by $\mathcal{O}_{x_1},\cdots, \mathcal{O}_{x_k}$ and for $i\neq j$ the points $[x_j]$ and $[x_j]$ lie on different leaves of the singular foliation associated with $\overline{\omega'}$. That is, it follows that $\overline{\omega'}$ is generic since the $G$-path integral $\int_{\sigma}\omega'\neq 0$ for any smooth $G$-path $\sigma$ from $x_j$ to $x_i$. Indeed,

\begin{eqnarray*}
\textnormal{Per}_\xi(\Pi_1^\textnormal{orb}(X))\ni\int_{\sigma^{-1}\ast \sigma_{ij}}\omega'&=&-\int_{\sigma}\omega'+\int_{\sigma_{ij}}\omega+\sum_{j=1}^{k}a_j(f_j(\sigma_{ij,n}(1))-f_j(\sigma_{ij,0}(0)))\\
&=& -\int_{\sigma}\omega'+\int_{\sigma_{ij}}\omega+a_i-a_j,
\end{eqnarray*}
which forces $\int_{\sigma}\omega'$ to be nonzero. Besides, $\overline{\omega}|_{U}=\overline{\omega'}|_{U}$ where $U=\bigcup U_j$, or, equivalently $\omega|_{\pi^{-1}(U)}=\omega'|_{\pi^{-1}(U)}$. 
\end{proof}

Additionally:

\begin{lemma}
Let $\overline{\omega}$ be a closed basic 1-form of Morse type on $X$ such that its cohomology class $[\omega]\in  H^1(X)$ is a multiple of an integral class. Then, all the leaves of the 
singular foliation $\tilde{\mathcal{F}}_\omega$ of $X$ are compact.
\end{lemma}

\begin{proof}
By \cite[Prop. 3.4]{V} it follows that there exists a basic smooth function $f:M\to S^1$ such that $\omega=af^{\ast}(d\theta)$, where $a\in \mathbb{R}$ is a constant and $d\theta$ is the angle form on $S^1$. Therefore, the leaves of $\mathcal{F}_\omega$ are the connected components of the inverse images of points $f^{-1}(z)$ with $z\in S^1$. Consequently, the induced leaves in $X$ are connected components of the inverse images of points $\overline{f}^{-1}(z)$ with $z\in S^1$ where $\overline{f}:X\to S^1$ is the induced continuous map. That is to say, the leaves of the singular foliation $\tilde{\mathcal{F}}_\omega$ in $X$ are compact. 
\end{proof}

Let us suppose that $\overline{\omega}$ generates a singular foliation $\tilde{\mathcal{F}}_\omega$ of $X$ with all leaves compact. We want to associate to $\overline{\omega}$ an oriented connected graph $\overline{\Gamma}_\omega$ in the following way. We say that two points in $X$ are \emph{equivalent} if they belong to the same leaf of the singular foliation $\tilde{\mathcal{F}}_\omega$.

\begin{lemma}\label{w-graph}
The singular leaves are isolated (i.e. each singular leaf has
a neighborhood, free of other singular leaves) and the quotient space $\overline{\Gamma}_\omega=X/\sim$ is a graph.
\end{lemma}
\begin{proof}
This follows by adapting the arguments in \cite[s. 5.2]{FKL} or \cite[step. (1)]{Honda} to the context described in Remark \ref{RemakrStructure} together with Lemmas \ref{Lemma:localReebStability} and \ref{lem:generic}. All zeros of index $j$ with $1<j<\dim X-1$ do not give rise to true vertices of the graph $\overline{\Gamma}_\omega$, since the surgeries corresponding to passing such zeros do not change the connectivity of the leaves. Nevertheless, unlike the manifold case, in our context there might be zeros of index $1$ of $\dim X-1$ for which surgeries corresponding to passing them do not change the connectivity of the leaves, so that those zeros do not give rise to true vertices of the graph $\overline{\Gamma}_\omega$ either, compare Figures \ref{Fig:FoliationinR3} and \ref{Fig:FoliationinR3Z2}. Hence, the vertices of $\overline{\Gamma}_\omega$ are determined by some of the zeros of index $1$ of $\dim X-1$ and by all the zeros of index $0$ and $\dim X$ which always produce terminal univalent vertices.
\end{proof}

In particular, if $\overline{\omega}$ is generic then $\overline{\Gamma}_\omega$ becomes a trivalent\footnote{Each vertex has exactly 3 edges.} directed graph. We refer the reader to \cite[step. (1)]{Honda} where the main features of these graphs are well-described and explored.

	\begin{remark}
		Let us consider the simple situation in which $X$ has the local quotient type $\mathbb{R}^3/G$, where $G$ is either $\{\textnormal{id}\}$ or else $\mathbb{Z}_2$, and the action of $\mathbb{Z}_2$ on $\mathbb{R}^3$ is given by reflecting along the $z$-coordinate. Assume that $\omega$ is locally given by $df$ with $f=x^2+y^2-z^2$, so that the leaves of the singular foliation are  parametrized by $f=t$, thus having a singular leaf when $t=0$. On the one hand, if $G=\{\textnormal{id}\}$ then the level sets have two $(G\ltimes \mathbb{R}^3)$-connected components if $t<0$ and one $(G\ltimes \mathbb{R}^3)$-connected component if $t>0$, see Figure \ref{Fig:FoliationinR3}. On the other hand, if $G=\mathbb{Z}_2$ then the level sets have always a single $(G\ltimes \mathbb{R}^3)$-connected component, see Figure \ref{Fig:FoliationinR3Z2}. Hence, the local picture for the graph associated with the singular foliation for the case $G=\{\textnormal{id}\}$ is the component defined by three edges with four vertices, while for $G=\mathbb{Z}_2$ it is the component defined by one edge with two vertices.
		\begin{figure}[H]
			\centering
			\includegraphics[width=0.3\textwidth]{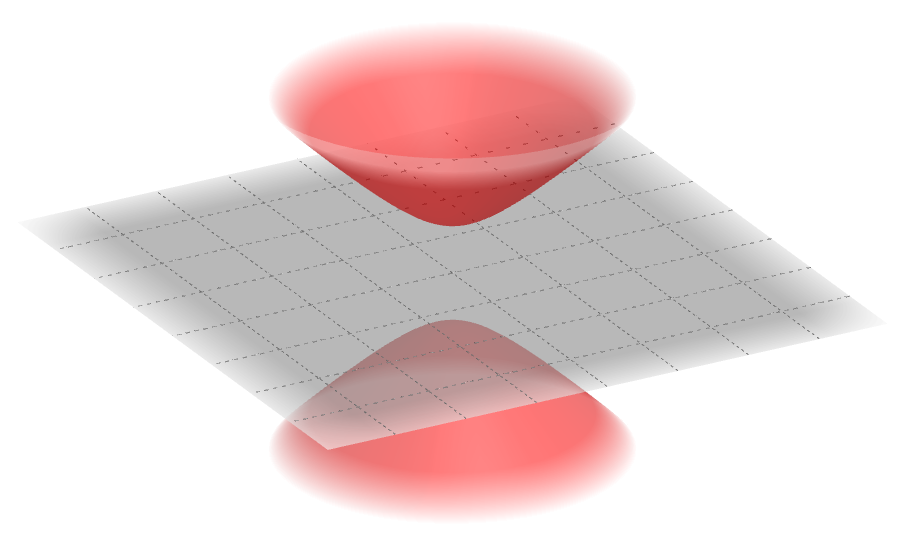}
			\includegraphics[width=0.3\textwidth]{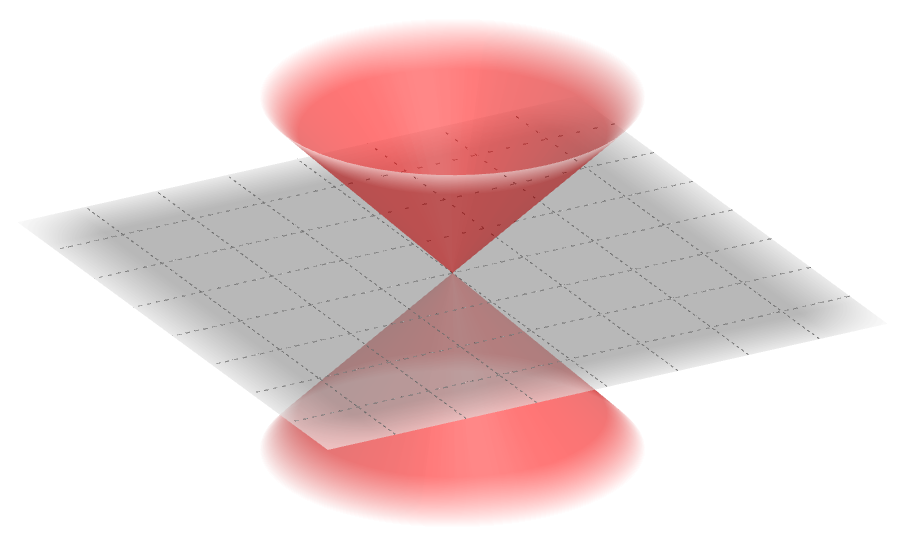}
			\includegraphics[width=0.3\textwidth]{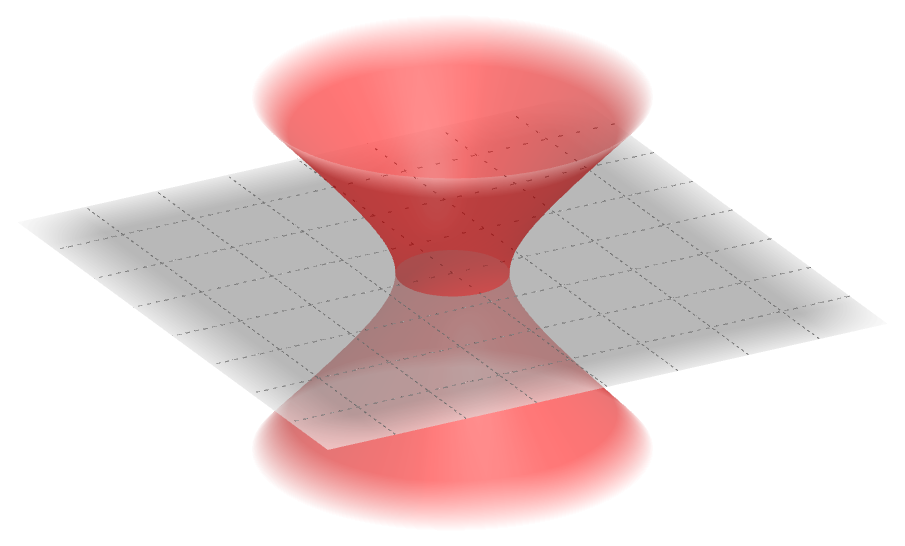}
			\caption{\footnotesize Singular foliation in $\mathbb{R}^3$ defined by $x^2+y^2-z^2=t$ at $t=-2, 0, 2$, respectively.}
			\label{Fig:FoliationinR3}
		\end{figure}
		\begin{figure}[H]
			\centering
			\includegraphics[width=0.3\textwidth]{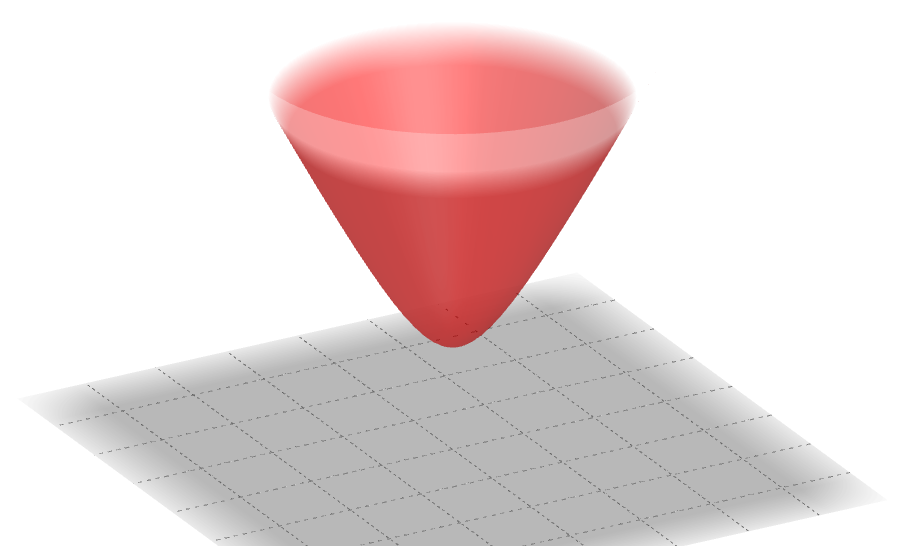}
			\includegraphics[width=0.3\textwidth]{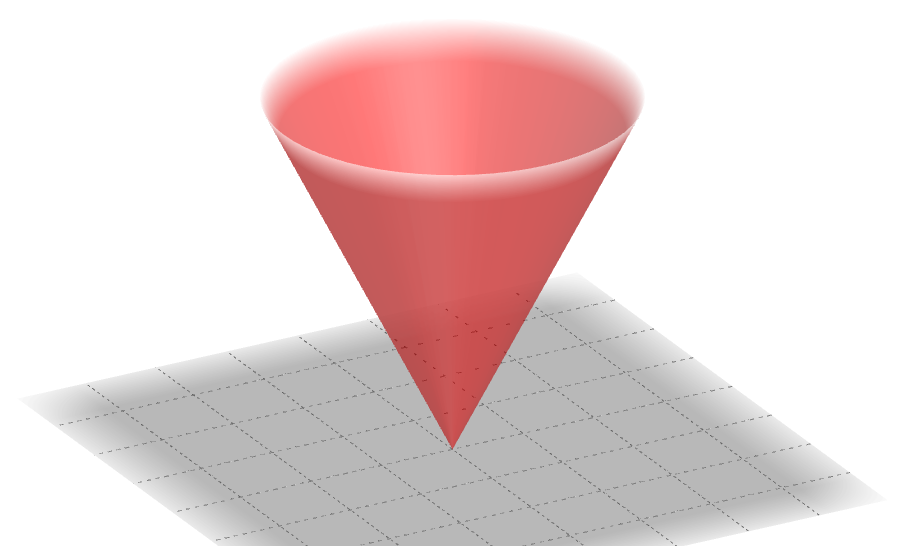}
			\includegraphics[width=0.3\textwidth]{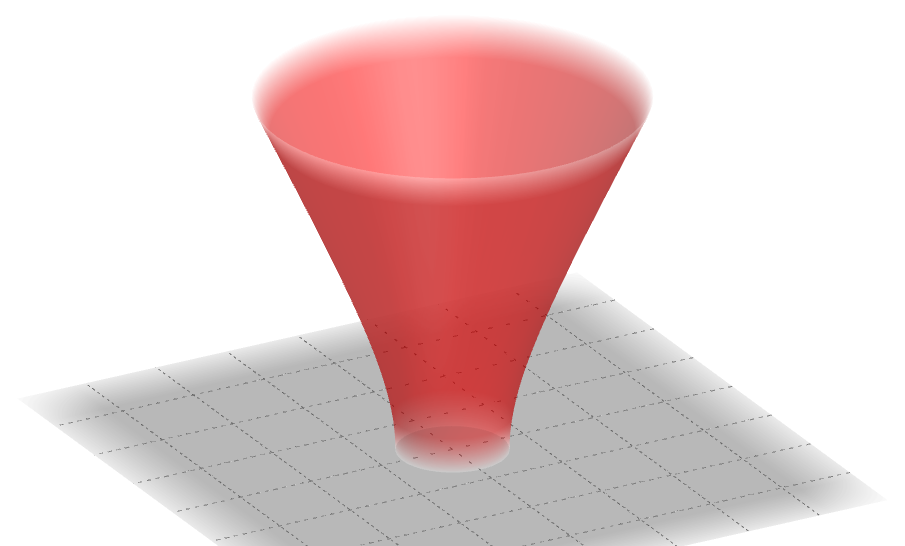}
			\caption{\footnotesize Singular foliation in $\mathbb{R}^3/\mathbb{Z}_2$ defined by $x^2+y^2-z^2=t$ at $t=-2, 0, 2$, respectively.}
			\label{Fig:FoliationinR3Z2}
		\end{figure}
	\end{remark}

We are now in position to start generalizing the main results in \cite{FKL} concerning the compactness and non-compactness of the leaves of the singular foliation $\tilde{\mathcal{F}}_\omega$ in $X$ (see Theorem \ref{thm1}). This can be addressed by following the strategy described in \cite[s. 8]{FKL} step by step, yet using the terminology and results developed in \cite{Hep,V} (see also \cite{OV}) for the case of orbifolds. Although the ideas of the proofs are natural and straightforward adaptations of the classical ones, we bring enough details in order to use most of the machinery previously introduced. 

\begin{proof}[Proof of item (1) in Theorem \ref{thm1}]

%\begin{theorem}
%A basic cohomology class $\xi\in H^1(X)$ has as a representative a closed basic 1-form $\omega$ of Morse type with all the induced leaves of the singular foliation $\tilde{\mathcal{F}}_\omega$
%compact in $X$ if and only if the homomorphism of periods $\Pi_1^{\textnormal{orb}}(X)\to (\mathbb{R},+)$ induced by $\xi$ can be factorized 
%as 
%$$\Pi_1^{\textnormal{orb}}(X)\to F\to (\mathbb{R},+),$$
%through a free group $F$. 
%\end{theorem}

 Let $\omega$ be a closed basic 1-form of Morse type representing $\xi$ and such that all the induced leaves $\overline{\mathcal{L}}$ of the singular foliation $\tilde{\mathcal{F}}_\omega$ in $X$ are compact. We can consider the graph $\overline{\Gamma}_\omega$ from  Lemma \ref{w-graph}. It is possible to assign positive weights to the edges of $\overline{\Gamma}_\omega$ as follows. Let $\pi_{\Gamma}:X\to \overline{\Gamma}_\omega$ be the canonical projection. If $\overline{\sigma}$ is an oriented edge of $\overline{\Gamma}_\omega$ from $p$ to $q$ then its \emph{weight} is defined to be the value of the $G$-path integral $\int_{\sigma}\omega$ of any $G$-path from $\psi^{-1}(p)$ to $\psi^{-1}(q)$ sitting inside $\psi^{-1}(\overline{\sigma})$. Here $\psi:= \pi_{\Gamma}\circ \pi$ stands for the composition $M\to X\to \overline{\Gamma}_\omega$. Thus, we get a function $W:\textnormal{edges}(\overline{\Gamma}_\omega)\to \mathbb{R}^+$ by assigning to each oriented edge its weight. Firstly, $\psi:M\to \overline{\Gamma}_\omega$ is a basic function since the whole groupoid orbit of a point in $M$ is inside the leaf of $\mathcal{F}_\omega$ passing through such a point, thus inducing a Lie groupoid morphism from $G\rr M$ to the unit Lie groupoid $\overline{\Gamma}_\omega\rr \overline{\Gamma}_\omega$. Secondly, $W$ is a 1-cocycle for the singular cohomology of $\overline{\Gamma}_\omega$, so that we get a cohomology class $\xi'=[W]\in H^1(\overline{\Gamma}_\omega,\mathbb{R})$ which satisfies $\psi^\ast (\xi')=\xi$ by construction. Therefore, with respect to any base-point $x_0\in M$, the homomorphism of periods $\textnormal{Per}_\xi:\Pi_1^{\textnormal{orb}}(X,[x_0])\to (\mathbb{R},+)$ can be factorized as 
$$\Pi_1^{\textnormal{orb}}(X,[x_0])\xrightarrow[]{\it \psi_\ast} \Pi_1(\overline{\Gamma}_\omega, \psi(x_0)) \xrightarrow[]{\it \textnormal{Per}_{\xi'}} (\mathbb{R},+).$$

\noindent Hence, $\Pi_1(\overline{\Gamma}_\omega, \psi(x_0))$ is the free group we were looking for.

Conversely, let $\xi\in H^1(X)$ be a cohomology class such that its homomorphism of periods can be factorized as $\Pi_1^{\textnormal{orb}}(X)\to F_r\to (\mathbb{R},+)$ through a free group $F_r$ of rank $r$. It follows that using a $G$-invariant version of the Thom--Pontrjagin construction (see \cite{FKL}), we may find a closed basic 1-form $\omega$ of Morse type on $M$ which presents the class $\xi$ and such that all the leaves of the singular foliation $\tilde{\mathcal{F}}_\omega$ of $X$ are compact. Let $W_r=\bigvee_{j=1}^r S^1_j$ denote a bouquet of $r$ circles. By averaging with respect to a proper Haar measure system on $G$ (see, e.g. \cite{CM}), it follows that there exist a cohomology class $\xi'\in H^1(W_r,\mathbb{R})$ and a continuous basic function $\lambda:M\to W_r$ such that $\lambda^\ast(\xi')=\xi$. Let us suppose that every circle $S_j^1$ has a fixed orientation, denote by $p_j=\langle \xi', [S^1_j]\rangle \in \mathbb{R}$, and pick points $a_j\in S^1_j$ for all $j=1,\cdots, r$. By performing a small perturbation if needed, we may assume that $\lambda$ is smooth near each fiber $\lambda^{-1}(a_j)$ and that each $a_j$ is a regular value. It follows that $\Lambda=s^\ast \lambda=t^\ast \lambda$ is also smooth around $\Lambda^{-1}(a_j)$ since $s$ and $t$ are local diffeomorphisms and $(G_j\rr M_j):=(\Lambda^{-1}(a_j)\rr \lambda^{-1}(a_j))$ is a Lie subgroupoid of $G\rr M$, since $\lambda$ is basic and $a_j$ is a regular value. Note that each $G_j\rr M_j$ keeps being proper and étale, so that it represents a sub-orbifold $X_j$ of $X$. We can model the orbifold normal bundle of $X_j$ as the vector bundle $\nu(X_j)$ over $X_j$ presented by the normal subgroupoid $\nu(G_j)\rr \nu(M_j)$, whose Lie groupoid structure is induced by that of the tangent groupoid $TG\rr TM$, compare \cite{dho}. The normal bundle of $\nu(M_j)$ is oriented by the fixed orientation of the circle $S^1_j$, so that $\nu(G_j)$ is also oriented, as its structural maps are fiberwise linear isomorphisms. In other words, $\nu(X_j)$ is oriented. We denote by $\overline{M}$ (resp. $\overline{G}$) the result of cutting $M$ open along the disjoint submanifolds $M_1, \cdots, M_r$ (resp. $G_1, \cdots, G_r$). The boundary $\partial \overline{M}$ (resp. $\partial \overline{G}$) is given by two copies $M_j^+$ and $M_j^-$ (resp. $G_j^+$ and $G_j^-$) for all $j=1, \cdots, r$, where the notation $M^{\pm}_j$ says that on $M^+_j$ the positive normal points inside $\overline{M}$, while on $M^-_j$ it points outside $\overline{M}$. By construction, we obtain a groupoid $\overline{G}\rr \overline{M}$ with boundary $\partial \overline{G}\rr \partial\overline{M}$ for which each copy satisfies $G^{\pm}_j\rr M^{\pm}_j$, as the structural maps of $G_j$ preserve the normal orientations. Furthermore, there is a canonical factor-morphism $\phi:(\overline{G}\rr \overline{M})\to (G\rr M)$ sending each Lie groupoid $G^{\pm}_j\rr M^{\pm}_j$ isomorphically onto $G_j\rr M_j$. These groupoid structures determine an orbifold $\overline{X}$ with boundary $\partial \overline{X}$ which consists of two copies $X^\pm_j$ for each $j=1, \cdots, r$ as well as a canonical factor map $\hat{\phi}:\overline{X}\to X$ sending each $X^\pm_j$ homeomorphically onto $X_j$. In consequence, $\hat{\phi}^\ast (\xi)=0$. 

As in the manifold case \cite[p. 128]{F}, and using the extension of Morse theory to orbifolds developed in \cite{Hep}, one can construct a Morse function $\overline{f}: \overline{X}\to \mathbb{R}$ with the following properties:
\begin{itemize}
\item the differential $d\overline{f}$ is nonzero on $X^\pm_j$ for all $j=1, \cdots, r$,
\item $\overline{f}$ assumes a constant value on each orbifold $X^\pm_1, \cdots, X^\pm_r$, and
\item $\overline{f}|_{X^-_j}-\overline{f}|_{X^+_j}=p_j$ and $(d\overline{f})|_{X^-_j}$ and $(d\overline{f})|_{X^+_j}$ match for all $j=1, \cdots, r$.
\end{itemize}

The last property implies that $d\overline{f}$ defines a closed 1-form $\overline{\omega}$ on $X$ satisfying $\hat{\phi}^\ast (\overline{\omega})=d\overline{f}$. It follows that $\overline{\omega}$ is of Morse type, has zero periods in $\overline{X}$, and has period $p_j$ along any closed $G$-path in $M$ which crosses $M_j$ once in the positive direction and does not intersect the submanifolds $M_k$ with $k\neq j$. Therefore, $\overline{\omega}$ is presented by a closed basic 1-form of Morse type $\omega$ on $M$ which in turn represents the cohomology class $\xi$. Additionally, it holds that for any leaf $\overline{\mathcal{L}}$ of the singular foliation $\tilde{\mathcal{F}}_\omega$ in $X$ either $\overline{\mathcal{L}}=X_j$ for some $j=1,\cdots, r$ or the preimage $\hat{\phi}^{-1}(\overline{\mathcal{L}})\subset \textnormal{int}(\overline{X})$ is a connected component of a level set $\overline{f}^{-1}(a)$. That is to say, all the leaves of the singular foliation $\tilde{\mathcal{F}}_\omega$ in $X$ are compact.
\end{proof}

Item (2) of Theorem \ref{thm1} explains how the compact and noncompact leaves of $\tilde{\mathcal{F}}_\omega$ co-occur in $X$ by describing a splitting of it into compact suborbifolds $X_c$ and $X_\infty$ which are roughly the union of the compact and noncompact leaves of $\tilde{\mathcal{F}}_\omega$, respectively. In order to do so, we need to introduce some additional notions. First, the \emph{rank} of a cohomology class $\xi\in H^1(X)$ is defined to be the rank of the image of its homomorphism of periods $\textnormal{Per}_\xi$, see \cite{V}. Second, if $\overline{\mathcal{L}}$ is a singular leaf of $\tilde{\mathcal{F}}_\omega$ then its \emph{singular leaf components} are defined to be the closures in $\overline{\mathcal{L}}$ of the connected components of $\overline{\mathcal{L}}\backslash  \lbrace [x]: \overline{\omega}([x])=0\rbrace$. Of course, we can similarly define the singular leaf components of $\mathcal{L}$ in $M$ by looking at the closure of its connected components of $\mathcal{L}\backslash  \lbrace \mathcal{O}_x: \omega(x)=0\rbrace$.

%\begin{theorem}\label{thm2}
%Let $\overline{\omega}$ be a closed 1-form of Morse type on a compact orbifold $X$ of dimension $n$. Then, $X$ is the union of two compact $n$-dimensional suborbifolds $X_c$ and $X_\infty$ with a common, possibly singular, boundary satisfying:
%\begin{enumerate}
%\item $\textnormal{int}(X_c)$ is the union of all of the compact leaves (nonsingular and singular) plus some compact singular leaf components of non-compact singular leaves,
%\item $\textnormal{int}(X_\infty)$ is the union of all noncompact nonsingular leaves, all noncompact singular leaf components and some compact singular leaf components of noncompact singular leaves,
%\item $\partial  X_c=\partial X_\infty=X_c\cap X_\infty$ is the union of finitely many compact singular leaf components of noncompact singular leaves. It is an orbifold except at finitely many points which are zeros of $\overline{\omega}$,
%\item if $X_\infty\neq \emptyset$ then the rank of the cohomology class $\xi_\infty=\xi|_{X_\infty}\in H^1(X_\infty)$ is greater than $1$. Here $\xi$ stands for $[\omega]\in H^1(X)$, and
%\item if $\overline{\omega}$ is generic then the boundary $\partial  X_c=\partial X_\infty=X_c\cap X_\infty$ is the union of all compact singular leaf components of non-compact singular leaves. It is a closed $(n-1)$-dimensional topological orbifold.
%
%\end{enumerate}
%\end{theorem}

\begin{proof}[Proof of item (2) in Theorem \ref{thm1}]
Suppose that $\overline{\omega}$ is presented by a closed basic $1$-form of Morse type $\omega$ on $M$. We denote by $U$ the union of all the compact leaves of $\tilde{\mathcal{F}}_\omega$, including the singular ones. Let us check that $U$ is open. If $\overline{\mathcal{L}}$ is a compact leaf then there are open sets $W$ and $V$ in $X$ such that $\overline{\mathcal{L}}\subset \textnormal{cl}(W)\subset V$ and $\overline{\omega}|_{V}=d\overline{f}$, where $\overline{f}:V\to \mathbb{R}$ is the induced function associated with a smooth basic function $f:\pi^{-1}(V)\to \mathbb{R}$. We may actually assume that $V$ is built so that any $G$-loop $\sigma$ in $\pi^{-1}(V)$ is homotopic to a $G$-loop in $\mathcal{L}$ and hence $\int_\sigma \overline{\omega}=0$, thus obtaining that $\omega|_{\pi^{-1}(V)}$ 
is the differential of a basic smooth function, compare \cite[Prop. 3.3]{V}. For sake of simplicity, one may also assume that $\overline{f}(\overline{\mathcal{L}})=0$ and that $V$ contains no zeros of $\overline{\omega}$ except possibly for those in $\overline{\mathcal{L}}$. Working locally around every point of $\overline{\mathcal{L}}$ and using the fact that  it is compact one can guarantee the existence of a certain $\epsilon>0$ such that $\overline{f}^{-1}(a)\subset W$ for all $a\in (-\epsilon,\epsilon)$. Since $f$ is basic, we actually have that $f^{-1}(a)$ is a saturated submanifold of $\pi^{-1}(V)$ lying in $\pi^{-1}(W)\subset \pi^{-1}(\textnormal{cl}(W))\subset \pi^{-1}(V)$, thus deducing that $\overline{f}^{-1}(a)$ is a compact suborbifold for all $a\in (-\epsilon,\epsilon)$. This shows that an open neighborhood of $\overline{\mathcal{L}}$ is the union of compact leaves of the singular foliation $\tilde{\mathcal{F}}_\omega$, so that $U$ is open. Hence, we define $X_c=\textnormal{cl}(U)$ and denote by $G_c\rr M_c$ the proper étale Lie groupoid presenting it. Let $Y\subset X$ denote the union of all compact leaves and all compact leaf components of noncompact singular leaves. As consequence of Remark \ref{RemakrStructure}, one has that $X_0=X\backslash \lbrace [x]: \overline{\omega}([x])=0\rbrace$ admits a decomposition by codimension 1 compact suborbifolds. Therefore, since $X$ is compact the orbifold version of Haefliger's result \cite[Thm. 3.2]{Haefliger} implies that the union $Y_0$ of all closed leaves in $X_0$ is closed. That is, $Y$ is closed. Hence, as $U$ is contained in $Y$ we have that $X_c\subset Y$ and its boundary $\partial X_c=X_c\backslash U$ is the union of finitely many compact leaf components of noncompact singular leaves. 

Let us now suppose that  $X_\infty\neq \emptyset$, so that there is a nonclosed leaf $\overline{\mathcal{L}}\subset \textnormal{int}(X_\infty)$ and a point $[y]\in \textnormal{cl}(\overline{\mathcal{L}})\backslash \overline{\mathcal{L}}$. Without loss of generality one may assume that $[y]$ is not a zero of $\overline{\omega}$. We know that there exist an open subset $\mathcal{O}_y\subset U\subset M$ and a basic smooth function $f_U:U\to \mathbb{R}$ such that $\omega|_U=df_U$. Thus, $\overline{\mathcal{L}}\cap \pi(U)$ can be viewed as the union of the set $\overline{f}_U=c_k$ for an infinite sequence $\lbrace c_k\rbrace\subset \mathbb{R}$ and $\overline{f}_U([y])\subset \textnormal{cl}(\lbrace c_k\rbrace)$. Given two sequence points $c_k$ and $c_l$ we consider a $G$-loop $\sigma_{k,l}$ in $\pi^{-1}(\textnormal{int}(X_\infty))$ starting at a point $y_l\in U$ with $f(y_l)=c_l$, traveling along the leaf $\mathcal{L}$ to a point in $y_k\in U$ with $f(y_k)=c_k$, and ending at the initial starting point $y_l\in U$. As consequence of the results shown in \cite[s. 3.1]{V}, we get that $\int_{\sigma_{k,l}}\overline{\omega}=c_l-c_k$, which ensures that one can find $G$-loops $\sigma$ in $\pi^{-1}(\textnormal{int}(X_\infty))$ whose $G$-path integral $\int_\sigma\overline{\omega}$ is as small as we want it to be. In other words, $\textnormal{rk}(\xi_\infty)>1$.

Finally, let $\overline{\omega}$ be generic. Suppose that a noncompact singular leaf 
$\overline{\mathcal{L}}$ is the union of two leaf components $C_1$ and $C_2$ such that $C_1$ is compact and $C_2$ is noncompact. If $[x]\in X$ is a zero of index $\lambda$ then by the Morse handle decomposition for orbifolds (see \cite{Hep,OV}), there is a neighborhood $W$ of $[x]$ such that $(W,W\cap \overline{\mathcal{L}})\cong$ cone over $(S^{n-1}/G_x,S^{\lambda-1}/G_x\times S^{n-\lambda-1}/G_x)$ if $0<\lambda<n-1$, while if $\lambda=1$ or $\lambda=n-1$, then $(W,W\cap \overline{\mathcal{L}})\cong$ cone over $(S^{n-1}/G_x, S^{n-2}/G_x)$ or $(W,W\cap \overline{\mathcal{L}})\cong$ cone over $(S^{n-1}/G_x, S^0/G_x\times S^{n-2}/G_x)$. It is worth saying that for this last statement to make sense we are identifying $S^{n-1}$ with the unit sphere $\partial D(\mathcal{O}_x)_x$ in $\nu_x(\mathcal{O}_x)$, as well as $S^{\lambda-1}$ and $S^{n-\lambda-1}$ with the unit spheres $\partial D_-(\mathcal{O}_x)_x$ and $\partial D_+(\mathcal{O}_x)_x$ in $\nu_-(\mathcal{O}_x)_x$ and $\nu_+(\mathcal{O}_x)_x$, respectively, see \cite[Prop. 5.8]{OV}. The intersection $C_1\cap C_2$ is a single point $[x]$ with $\omega(x)=0$ which must be of Morse index $1$ or $n-1$. Therefore, one can assume that there exists a neighborhood $W$ of $[x]$ such that $(W,W\cap \overline{\mathcal{L}})\cong$ cone over $(S^{n-1}/G_x, S^0/G_x\times S^{n-2}/G_x)$, where $W\cap C_1$ is one of the two possible cones over $ S^{n-2}/G_x$. In consequence, by mimicking the analysis of the boundary $\partial X_c$ as in the manifold case (compare \cite[p. 129]{F}), it follows that $C_1\subset X_c\cap X_\infty$, as desired.
\end{proof}

As an interesting outcome, one can get another Reeb-like stability theorem for our foliations with singularities.

\begin{corollary}\label{cor1}
Let $\overline{\omega}$ be a closed 1-form of Morse type on a compact and connected orbifold $X$ with no zeros of indices $1$ and $n-1$. Then, either 
any leaf of the singular foliation $\tilde{\mathcal{F}}_\omega$ is compact or any leaf is dense in $X$. If $\xi=[\omega]\in H^1(X)$ then the first case happens if and only if $\textnormal{rk}(\xi)\leq 1$ and the second case happens if and only if $\textnormal{rk}(\xi)>1$.
 \end{corollary}

\begin{proof}
If $\overline{\omega}$ has no zeros of indices $1$ and $n-1$ then the boundary $X_c\cap X_\infty$ is empty, so that either $X_c=\emptyset$ (all leaves are noncompact) or else $X_\infty=\emptyset$ (all leaves are compact). We only have to prove that $\textnormal{rk}(\xi)\leq 1$ if all leaves are compact, see item (d) in Theorem \ref{thm1}. By Lemma \ref{w-graph} we obtain that $\overline{\Gamma}_\omega$ is homemorphic to a circle or a closed interval, as all of its vertices are bivalent or 
univalent in this case. Hence, the results follows by applying Theorem \ref{thm1}.
\end{proof}

\section{Intrinsically harmonic closed 1-forms}\label{sec:4}

The aim of this section is to extend to the realm of orbifolds the celebrated Calabi's theorem which characterizes intrinsically closed harmonic 1-forms of Morse type in pure topological terms \cite{Calabi}. Most of the fundamental material about the differential geometry and topology of orbifolds that we will be using throughout this section can be found in references as \cite{AMR,FarProSea,KleinerLott}, but we refer the reader to the nice survey \cite{Caramello} and its quoted references. 

Recall that a Riemannian metric on $X$ can be thought of as an equivalence class of a groupoid Riemannian metric on $G\rr M$ in the sense of del Hoyo and Fernandes \cite{dHF,dHF2}. The manifolds of arrows $G$ and objects $M$ respectively inherit Riemannian metrics $\eta^{(1)}$ and $\eta^{(0)}$ such that $s$ and $t$ become local isometries and the inversion $i$ gives rise to an isometry. Additionally, the induced bundle metrics along the normal directions to the orbits are such that the normal isotropy representations are by linear isometries. This provides us with a way of inducing ``inner products'' over the ``coarse'' tangent spaces $T_{[x]} X \approx \nu_x(\mathcal{O}_x)/G_x$ for all $x\in M$. In this scenario, one can assume that the Laplace operators of $\eta^{(1)}$ and $\eta^{(0)}$ exist and $s, t$-commute (see \cite{BBB,Watson}), so that the Laplacian operator of $\eta^{(0)}$ restricts to the spaces of basic differential forms $\Omega_{\textnormal{bas}}^\bullet(G)$, thus providing a way to determine when a closed 1-form on $X$ is harmonic. A closed 1-form $\overline{\omega}$ on $X$ is called \emph{intrinsically harmonic} if it is harmonic with respect to some Riemannian metric on $X$. Equivalently, the closed basic 1-form $\omega$ on $M$ presenting $\overline{\omega}$ is harmonic with respect to some groupoid Riemannian metric on $G\rr M$.

A smooth $G$-path $\sigma=\sigma_ng_n\sigma_{n-1}\cdots \sigma_1g_1\sigma_0$ in $M$ is said to be $\omega$-\emph{positive} if $\omega(\dot{\sigma}_k(\tau))>0$ for all $k=0,1,\cdots,n$ and $\tau\in [0,1]$. Geometrically, this means that the velocity vector $\dot{\sigma}_k(\tau)$ of each path $\sigma_k(\tau)$ always points in the direction in which $\overline{\omega}$ increases (i.e. locally, where $f_U$ increases).

\begin{definition}
A closed 1-form $\overline{\omega}$ on $X$ is called \emph{transitive} is for any point $[x]\in X\backslash \textnormal{zeros}(\overline{\omega})$ there exists an $\omega$-positive smooth $G$-loop at $x\in M$. 
\end{definition}

Before stating and showing the main result of this section we exhibit some results exploring the notions we have just introduced. These are also straightforward generalizations of the corresponding results proven in \cite{FKL} for the manifold case. Recall that we have a canonical decomposition $X= X_c\cup X_\infty$ provided by item (2) in Theorem \ref{thm1}. First of all, it holds that the points in $X_\infty$ represent no difficulty for transitivity, since such a condition is automatically satisfied:

\begin{proposition}\label{pro:Harmonic1}
Let $\overline{\omega}$ be a generic closed 1-form of Morse type on $X$. Then for any pair of points $[x], [y]$, which are not zeros of $\overline{\omega}$, lying in the same $G$-connected component of $\textnormal{int}(X_\infty)$, there exists an $\omega$-positive smooth $G$-path from $x$ to $y$. In particular, there exists an $\omega$-positive smooth $G$-loop through any point $[x]\in \textnormal{int}(X_\infty)$ with $\omega(x)\neq 0$. 
\end{proposition}
\begin{proof}
Let us denote by $\textnormal{int}(M_\infty)$ the open subset in $M$ given by $\pi^{-1}(\textnormal{int}(X_\infty))$. This clearly determines an open Lie subgroupoid $\textnormal{int}(G_\infty)\rr \textnormal{int}(M_\infty)$ of $G\rr M$ presenting $\textnormal{int}(X_\infty)$. We start by fixing a point $[x]\in \textnormal{int}(X_\infty)$ which is not a zero of $\overline{\omega}$ and considering the \emph{positive upland} $U_x$ in  $\textnormal{int}(X_\infty)$. That is, the set defined as the points $[y]\in X$ such that $y\in \textnormal{int}(M_\infty)$ is the endpoint of an $\omega$-positive smooth $G$-path starting at $x$. Note that if $y$ can be reached starting from $x$ by moving along an $\omega$-positive smooth $G$-path then the same is true for any other point lying in the same leaf $\mathcal{L}$ as $y$. Hence, $U_x$ is open and is given as a union of leaves of the singular foliation $\tilde{\mathcal{F}}_\omega$. Additionally, using the fact that $\partial X_\infty$ is the union of all compact singular leaf components of noncompact leaves (see item (e) in Theorem \ref{thm1}) and arguing exactly as in the proof of \cite[Thm. 9.13]{F} it is simple to check that $\textnormal{cl}(U_x)\backslash  U_x$ is contained in $\lbrace [y]\in X_\infty: \omega(y)=0\rbrace\cup \partial X_\infty$. This implies that $U_x$ agrees with the $G$-connected component of $\textnormal{int}(X_\infty)$ by Theorem \ref{thm1}. So, the result holds true as claimed. 
\end{proof}

\begin{proposition}\label{pro1}
Let $\overline{\omega}$ be a transitive closed 1-form of Morse type on $X$ such that the associated singular foliation $\tilde{\mathcal{F}}_\omega$ has a compact leaf. Then, there exists a nonzero integral cohomology class $\theta\in H^1(X,\mathbb{Z})$ such that $\theta \cup \xi$ vanishes in $H^2(X)$. Here $\xi$ stands for the cohomology class $[\omega]\in H^1(X)$.
\end{proposition}

\begin{proof}
Since $\tilde{\mathcal{F}}_\omega$ has a compact leaf we get that there exist nonsingular compact leaves $\overline{\mathcal{L}}$ in $X$, see item (a) in Theorem \ref{thm1}. Note that the leaf $\overline{\mathcal{L}}$ does not bound in $X$, because the transitivity condition would be violated otherwise. There exists an open neighborhood $U$ of $\overline{\mathcal{L}}$ and a basic smooth function $f_U:\pi^{-1}(U)\to \mathbb{R}$ such that the induced function $\overline{f}:U\to \mathbb{R}$ has no critical points and $\omega|_{\pi^{-1}(U)}=df_U$, see Lemma \ref{Lemma:localReebStability}. For simplicity, we may assume that $\overline{f}^{-1}(0)=\overline{\mathcal{L}}$ and there is an $\epsilon>0$ small enough such that the level set $\overline{f}^{-1}(\tau)\cong \overline{\mathcal{L}}$ for all $\tau \in (-\epsilon,\epsilon)$. Let us consider a smooth function $\rho:(-\epsilon,\epsilon)\to \mathbb{R}$ such that $\rho(\tau)$ equals $\tau$ if $\vert \tau \vert <\epsilon/2$, $1/2$ if $\tau\in (3/4\epsilon,\epsilon)$, and $-1/2$ if $\tau\in (-\epsilon,-3/4\epsilon)$. We define $\alpha|_{\pi^{-1}(U)}=d(\rho\circ f_U)$ and $\alpha|_{M\backslash \pi^{-1}(U)}=0$.  This is clearly the closed basic 1-form on $M$ which represents the integral nonzero cohomology $\theta\in H^1(X,\mathbb{Z})$ we are looking for, compare \cite[Prop. 9.14]{F}.
\end{proof}

It is well-known that the singular homology groups $H_\bullet(X,\mathbb{Z})$ of $X$ can be described in terms of the simplicial structure of the nerve of $G\rr M$, more precisely, as their associated total homology groups \cite{Be}. 

\begin{corollary}\label{cor2}
Suppose that $X$ is a compact orbifold such that the singular homology $H_1(X,\mathbb{Z})$ has no torsion and the cup-product pairing $H^1(X)\times H^1(X)\to H^2(X)$ is nondegenerate\footnote{That is, for each nonzero $\xi\in H^1(X)$ there exists $\eta\in H^1(X)$ such that $\xi \cup \eta\neq 0$.}. Let $\xi \in H^1(X)$ be a cohomology class with $\textnormal{rk}(\xi)=b_1(X)$, the first Betti number of $X$. Then, any transitive closed $1$-form of Morse type $\overline{\omega}$ representing $\xi$ is such that all the leaves of $\tilde{\mathcal{F}}_\omega$ are noncompact. 
\end{corollary}
\begin{proof}
This follows by arguing exactly as in the proof of \cite[Thm. 9.15]{F}.
\end{proof}

We can associate an oriented graph with any closed 1-form on $X$ such that its transitivity is reflected in the properties of such a graph. 

\begin{definition}
	Let $\Gamma$ be an oriented connected graph such that:
	\begin{enumerate}
		\item if $x,y\in \Gamma$ are vertices then there is a path from $x$ to $y$ that traverses edges of $\Gamma$ only in the positive direction, and 
		\item for any vertex $x\in \Gamma$ there is a closed path through $x$ that traverses edges of $\Gamma$ only in the positive direction.
	\end{enumerate}
	
	A graph $\Gamma$ satisfying both conditions is called \emph{Calabi graph}.
\end{definition}

It is clear that condition (1) implies (2). In fact, these items are actually equivalent as shown for instance in \cite{FKL}. If $\overline{\omega}$ is a closed 1-form of Morse type on $X$ then we can associate a graph $\overline{\Gamma}_\omega$ by mimicking the prescription provided in Lemma \ref{w-graph}. Namely, it is given by the quotient space $X/\sim$, where $[x]\sim [y]$ if and only if they lie on the same compact leaf of the singular foliation $\tilde{\mathcal{F}}_\omega$ or they belong to the same $G$-connected component of $X_\infty$. On the one hand, points of $\overline{\Gamma}_\omega$ corresponding to the $G$-connected components of $X_\infty$ are called \emph{special vertices}. Near a point representing a nonsingular compact leaf the graph $\overline{\Gamma}_\omega$ is locally homeomorphic to an interval. On the other hand, each compact leaf of $\tilde{\mathcal{F}}_\omega$ is represented by an \emph{ordinary 
point} of $\overline{\Gamma}_\omega$. The vertices of $\overline{\Gamma}_\omega$ associated to compact leaves are described in the proof of Lemma \ref{w-graph}. 

As an immediate consequence of Theorem \ref{thm1}, one concludes that:

\begin{proposition}
A generic closed 1-form of Morse type $\overline{\omega}$ is transitive if and only if the oriented graph $\overline{\Gamma}_\omega$ is transitive. 
\end{proposition}

We can now prove the main result of this section (see Theorem \ref{thm3}), concerning an extention of the celebrated Calabi's theorem (consult \cite{Calabi}), to the realm of orbifolds. 

\begin{proof}[Proof of Theorem \ref{thm3}]
One can follow the strategy used to prove this result in the case of manifolds (see \cite[Thm. 9.11]{F}), together with some of the tools for the Hodge star operator studied in \cite{Chian,FarProSea} for Riemannian orbifolds. The arguments for the proof in our case are similar, so we shall only sketch its main ideas for the sake of completeness. Let us start by assuming that $X$ is orientable, so that the normal action of each isotropy $G_x\curvearrowright \nu_x(\mathcal{O}_x)$ is orientation-preserving.

Suppose that there is a Riemannian metric on $X$ for which $\overline{\omega}$ is intrinsically harmonic. By contradiction, let us further assume that $\overline{\omega}$ is not transitive, so that there is a point $[x]\in X$ such that $\omega(x)\neq 0$ and there exists no $\omega$-positive smooth $G$-loop through $x$. Without loss of generality, 
we may assume that the leaf $\overline{\mathcal{L}}$ of the singular foliation $\tilde{\mathcal{F}}_\omega$ containing $[x]$ is nonsingular, thus obtaining that the arguments used the proof of Proposition \ref{pro:Harmonic1} guarantee that $\overline{\mathcal{L}}$ must be compact. If not, an $\omega$-positive smooth $G$-loop through $x$ would exist. 

Let $U_x$ be the positive upland at $[x]$. Recall that it is defined to be the set of all points $[y]\in X$ such that $y\in M$ is reachable starting from $x$ by an $\omega$-positive smooth $G$-path. It is clear that the space of all $y\in M$ satisfying the latter condition is saturated, so that it determines a subgroupoid of $G\rr M$. More importantly, the closure $Y$ of $U_x$ in $X$ gives rise to a compact orbifold with boundary. From the very definition of $U_x$ it follows that $\overline{\mathcal{L}}\subset \partial Y$. Besides, for any point $[y]\in \partial Y$ and for any tangent vector $[v]\in T_{[y]}X\approx \nu_y(\mathcal{O}_y)/G_y$ pointing inside $Y$ we get that $\overline{\omega}([v])=\omega(v)>0$, as the normal actions of the isotropies are orientation-preserving. Therefore, the orientation on $X$ induces an orientation on the boundary $\partial Y$ for which $(\star \overline{\omega})|_{\partial Y}>0$. Here $\star$ stands for the Hodge star operator associated with both the Riemannian metric and orientation on $X$. Using the orbifold version of the Stokes formula one produces a contradiction since $\int_{Y}d(\star \overline{\omega})=0$, as $\overline{\omega}$ is harmonic. 

Conversely, let us suppose that $\overline{\omega}$ is transitive, so that for any point $[x]\in X\backslash \textnormal{zeros}(\overline{\omega})$ there exists an $\omega$-positive smooth $G$-loop $\sigma_x$ at $x\in M$. We may think of $\sigma_x$ as an actual loop $\pi\circ \sigma_x$ within $X$ in such a way its image determines an orbifold embedding \cite{BorzeBrun,ChoHongShin}. As in the manifold case, for $n>2$ the last assertion can be obtained after applying a small perturbation 
and for $n = 2$ one may need to change the loop by performing simple modifications near the double points, where the tangent vectors might be not defined. Using the orbifold tubular neighborhood theorem (i.e. linearization \cite{dHF2}), orbifold partitions of unity (see \cite{KleinerLott,CM}), and carefully adapting the ideas in the proof of \cite[Thm. 9.11]{F}, one can build a closed differential form $\overline{\psi}\in \Omega^{n-1}(X)$, as well as a Riemannian metric on $X$ such that:
\begin{itemize}
\item $\overline{\omega}\wedge \overline{\psi}\geq 0$ (the positivity is understood with respect to the specified orientation on $X$),
\item $(\overline{\omega}\wedge \overline{\psi})([x])=0$ if and only if $[x]\in X$ is a zero of $\overline{\omega}$, and
\item for any point $[x_j]\in \textnormal{zeros}(\overline{\omega})=\lbrace [x_1],\cdots, [x_k]\rbrace$ there exists an open neighborhood $U_j$ such that the Riemannian metric restricted to it verifies $\overline{\psi}|_{U_j}=\star (\overline{\omega}|_{U_j})$.
\end{itemize}

More importantly, one has that $\overline{\psi}=\star \overline{\omega}$ and hence $d(\star \overline{\omega})=0$, as desired. 

Finally, if $X$ is not necessarily orientable	 then we use the orbifold orientable bundle $o(X)\cong \det (T^\ast X)\cong \wedge^n(T^\ast X)$ as done in \cite[s. 9.4.4]{F}.
\end{proof}

It is worth stressing that building upon the works \cite{FKL,Honda} several other interesting results concerning the Morse-theoretic properties of harmonic closed 1-forms on compact orbifolds can be established by employing similar ideas. Such results may be used to determine whether one can improve the topological lower bounds for the number of zeros of a closed 1-form of Morse type on a compact orbifold. The first attempt in that direction is provided by the Novikov type inequalities, which were recently established in \cite{V}.

\section{Examples}\label{sec:5}

Motivated by the constructions developed in \cite{FKL}, in this short section we provide some examples which allow us to illustrate our main results.

\begin{example}
	Suppose that $M$ is a closed manifold with abelian non-cyclic fundamental group. Let $\mathcal{F}$ be any proper foliation on $M$ with all principal leaves being 1-connected and let $X$ denote the compact orbifold presented by the holonomy groupoid $\textnormal{Hol}(M,\mathcal{F})\rr M$. From Remark \ref{remark1} it follows that $\Pi_1^{\textnormal{orb}}(X)\cong \Pi_1(M)$. If $\xi\in H^1(X)$ is a cohomology class such that its homomorphism of periods $\textnormal{Per}_\xi$ is a monomorphism then it can not be represented by a closed 1-form of Morse type on $X$ with all the leaves of the singular foliation $\tilde{\mathcal{F}}_\omega$ being compact. This is because $\textnormal{Per}_\xi$ can not be factorized through a free group, as any abelian subgroup of a free group is cyclic. 
	
	The same conclusion holds true in general if we consider cohomology classes for compact orbifolds having abelian non-cyclic orbifold fundamental group and for which the homomorphism of periods is a monomorphism.
\end{example}

One can obtain more explicit and interesting examples by exploiting the 2-torus and the compact 2-orbifold known as the \emph{pillowcase}. To do this, we need to briefly introduce the orbifold analogue of the connected sum constructions explored in \cite{FKL}. Such general connected sum operations accept as input two compact two-dimensional orbifolds $X_1, X_2$ with closed 1-forms of Morse type $\overline{\omega}_1, \overline{\omega}_2$ and constructs a closed 1-form of Morse type $\overline{\omega}$ on the connected sum $X:=X_1\sharp X_2$. Connected sums of compact orbifolds have been used for instance in \cite{FarProSea,KleinerLott}. The construction we want to describe is such that $\overline{\omega}|_{X_j\backslash D_j}=\overline{\omega_j}|_{X_j\backslash D_j}$ for $j=1,2$, where $D_j$ is any (small enough) disk in $X_i$ containing no critical nor orbifold singular points of $X_i$ and within which the connected sum will be constructed. Furthermore, the boundary $\partial D_i$ is not contained in any leaf of the foliation $\ker(\overline{\omega}_i)$.

Let us choose functions $\overline{f}_j:D_j\to \mathbb{R}$ such that $\overline{\omega}_j|_{D_j}=d\overline{f}_j$, compare \cite[Prop. 3.3]{V}. One can identity $D_j$ with an open rectangle $(a_j, b_j)\times (c_j, d_j)$ in $\mathbb{R}^2$, so that $\overline{f}_1(x,y)=y$. This enables us to move the rectangle by any translation in $\mathbb{R}^2$. We can now perform a connected sum of
$D_1$ and $D_2$ ambiently by connecting them with a straight tube in $\mathbb{R}^3$ in three different ways (see \cite[Fig. 9]{FKL}):

\begin{enumerate}
\item[(a)] $\overline{f}_1(D_1)\cap \overline{f}_2(D_2)=J$ is nonempty, the connecting tube intersects $D_j$ inside $\overline{f}_j^{-1}(J)$, and is
approximately horizontal,
\item[(b)] $\overline{f}_1(D_1)\cap \overline{f}_2(D_2)$ is empty, or
\item[(c)] $\overline{f}_1(D_1)\cap \overline{f}_2(D_2)=J$ is nonempty and the connecting tube has its two critical points at the same level.
 \end{enumerate}

We define $\overline{f}:D_1\sharp D_2\to \mathbb{R}$ to be the restriction of the height function, so that $d\overline{f}$ blends with $\overline{\omega}_1$ and $\overline{\omega}_2$, thus defining a closed 1-form on the connected sum $X$. It is important to notice that each of these constructions introduces two new critical points $[x], [y]$ of index 1. Additionally, $\overline{f}([x])<\overline{f}([y])$ in (a), $\overline{f}([x])>\overline{f}([y])$ in (b), and $\overline{f}([x])=\overline{f}([y])$ in (c).

If $\overline{\omega}_1$ and $\overline{\omega}_2$ are both transitive then so is $\overline{\omega}$ under construction (a) and every leaf of $\tilde{\mathcal{F}}_\omega$ intersects at least one leaf of $\tilde{\mathcal{F}}_{\omega_1}$ and one leaf of
$\tilde{\mathcal{F}}_{\omega_2}$. Under construction (b), there are new compact leaves produced and under construction
(c) there is just one new compact singular leaf component produced.

\begin{example}[Pillowcase]
	Consider the 2-torus $\mathbb{T}^2$ embedded in $\mathbb{R}^3$ together with the action by the group $\mathbb{Z}_2$ given by rotations of angle $\pi$ around one of the axes of $\mathbb{T}^2$. The quotient space $\mathbb{T}^2/\mathbb{Z}_2$ is an orbifold whose underlying space is homeomorphic to the 2-sphere and whose singular locus consists of four singular points each with isotropy $\mathbb{Z}_2$. As a consequence of  Remark \ref{remark1}, it follows that the orbifold fundamental group of $\mathbb{T}^2/\mathbb{Z}_2$ is a semi-direct product $\Pi_1^{\textnormal{orb}}(\mathbb{T}^2/\mathbb{Z}_2)\cong \mathbb{Z}^2\rtimes \mathbb{Z}_2$. The standard height function $h$ on $\mathbb{T}^2$ is invariant under the $\mathbb{Z}_2$-action, so that it descends to define an orbifold Morse function $\overline{h}:\mathbb{T}^2/\mathbb{Z}_2\to \mathbb{R}$ with one critical point $[s]$ of index data $(0,\mathbb{Z}_2)$, two critical points $[r],[q]$ of index data $(1,\mathbb{Z}_2)$, and one critical point $[p]$ of index data $(2,\mathbb{Z}_2)$. The foliations on $\mathbb{T}^2$ and $\mathbb{T}^2/\mathbb{Z}_2$ determined by the level set of $h$ and $\overline{h}$ are respectively depicted in Figure \ref{Fig:Pillow0}.
\end{example}
\begin{figure}[H]
	\centering
	\includegraphics[width=0.5\textwidth]{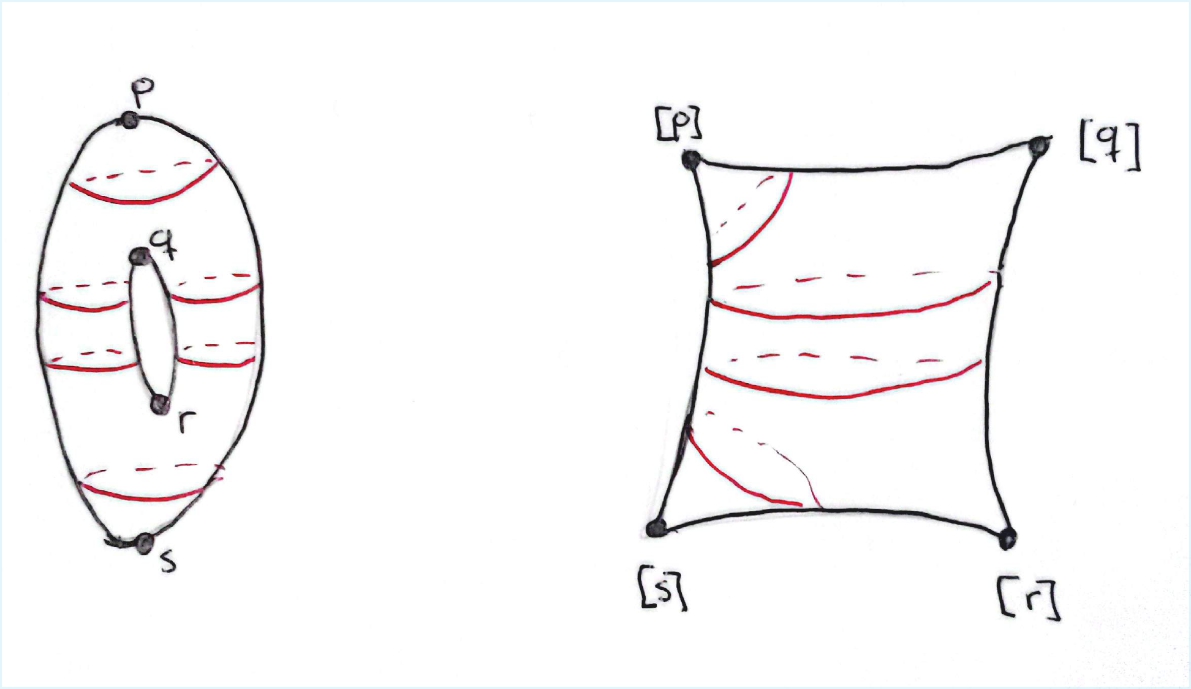}
	\caption{\footnotesize Foliation on the pillowcase.}
	\label{Fig:Pillow0}
\end{figure}
In next examples the pillowcase orbifold $\mathbb{T}^2/\mathbb{Z}_2$ is denoted by $Q$. We use $\theta$ and $\phi$ to denote the usual angle coordinates on $\mathbb{T}^2$, thus the corresponding angle forms are denoted by $d\theta$ and $d\phi$.

\begin{example}
Suppose that $X_1 = Q$ and $X_2 = \mathbb{T}^2$. Let $\overline{\omega}_1 = d\overline{h}$ and $\overline{\omega}_2 = p\,d\theta + q\,d\phi$, where $p,q$ are coprime integers. Under construction (a), the resulting closed 1-form $\overline{\omega}$ is transitive, and the singular foliation $\tilde{\mathcal{F}}_\omega$ has only compact leaves. Moreover, the period map factors through $\mathbb{Z}$.
See Figure \ref{Fig:Pillow1}.
\end{example}
\begin{figure}[H]
	\centering
	\includegraphics[width=0.4\textwidth]{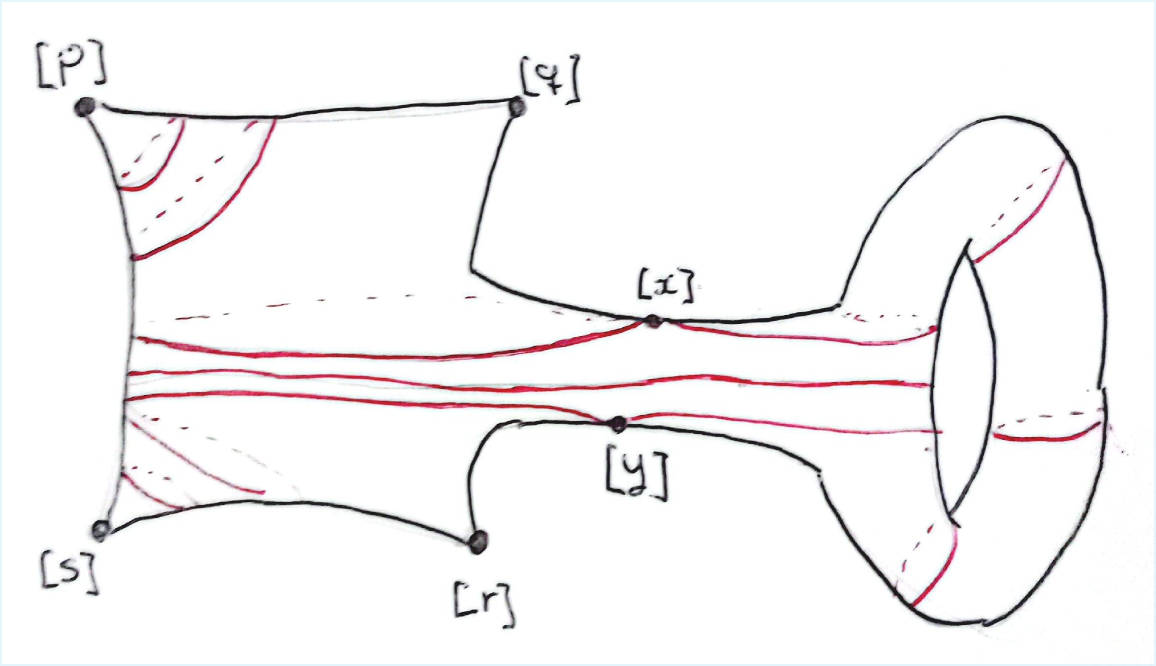}
	\caption{\footnotesize Connected sum of the pillowcase with the 2-torus under construction (a).}
	\label{Fig:Pillow1}
\end{figure}

\begin{example}
Suppose that $X_1 = Q$ and $X_2 = \mathbb{T}^2$. Let $\overline{\omega}_1 = d\overline{h}$ and $\overline{\omega}_2 = p\,d\theta + q\,d\phi$, where $p,q$ are linearly independent over $\mathbb{Q}$. Under construction (b), the resulting closed 1-form $\overline{\omega}$ is non-transitive, and the singular foliation $\tilde{\mathcal{F}}_\omega$ has both compact and noncompact leaves. Moreover, $X_c \cap X_{\infty}$ is a circle. See Figure \ref{Fig:Pillow2}.
\end{example}
\begin{figure}[H]
	\centering
	\includegraphics[width=0.4\textwidth]{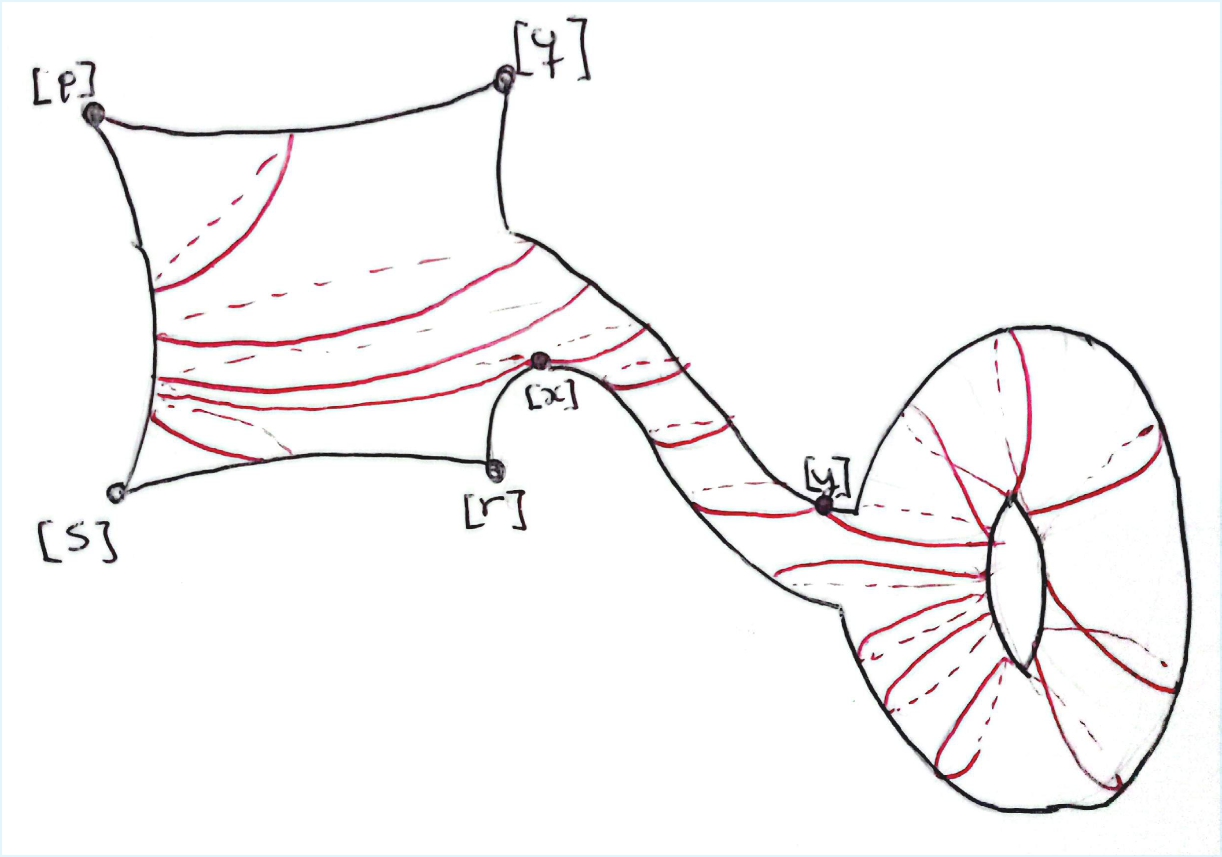}
	\caption{\footnotesize Connected sum of the pillowcase with the 2-torus under construction (b).}
	\label{Fig:Pillow2}
\end{figure}

\begin{example}
Suppose that $X_1 = Q$ and $X_2 = \mathbb{T}^2$. Let $\overline{\omega}_1 = d\overline{h}$ and $\overline{\omega}_2 = p\,d\theta + q\,d\phi$, where $p,q$ are linearly independent over $\mathbb{Q}$. Under construction (c), the resulting closed 1-form $\overline{\omega}$ is non-transitive, and the singular foliation $\tilde{\mathcal{F}}_\omega$ has both compact and noncompact leaves. Moreover, $X_c \cap X_{\infty}$ is a bouquet of two circles. See Figure \ref{Fig:Pillow3}.
\end{example}
\begin{figure}[H]
	\centering
	\includegraphics[width=0.4\textwidth]{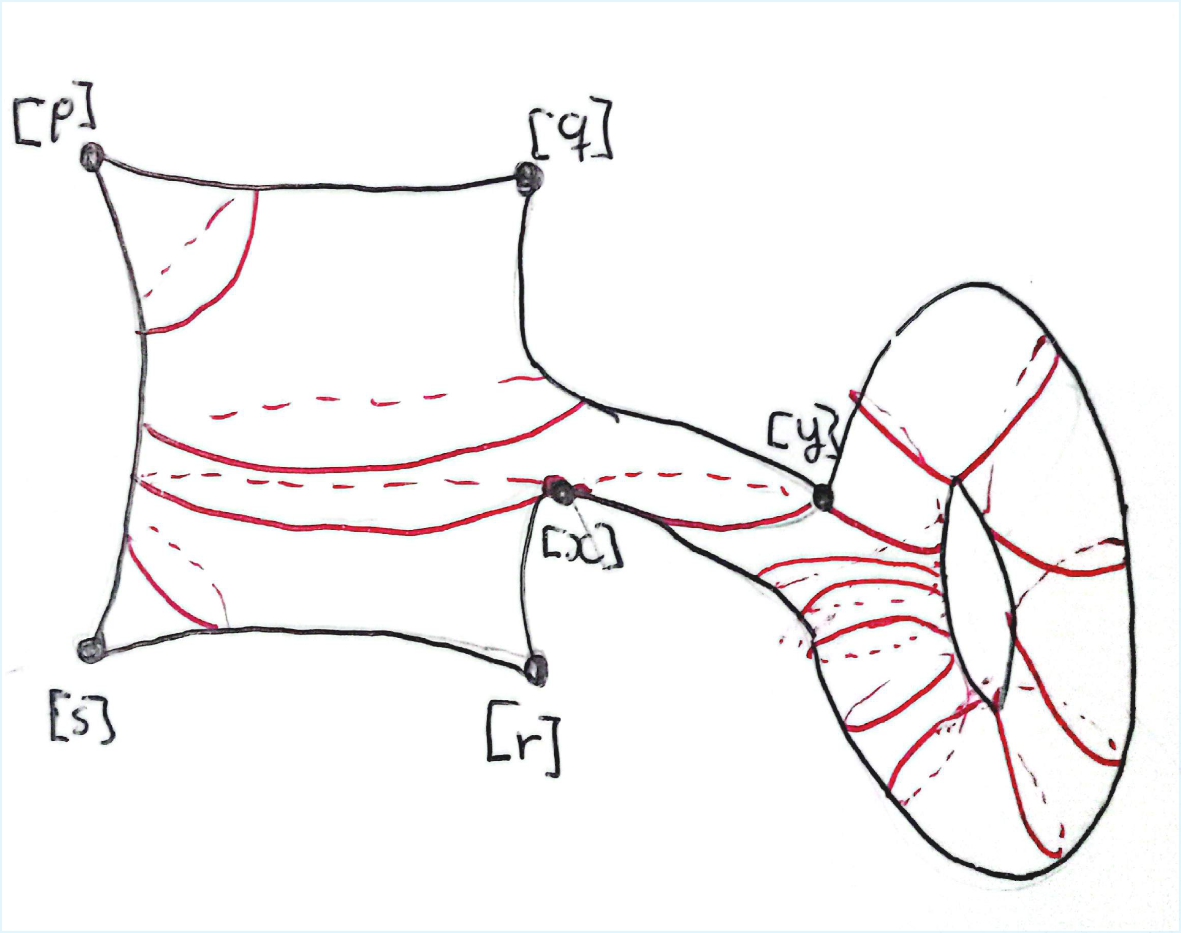}
	\caption{\footnotesize Connected sum of the pillowcase with the 2-torus under construction (c).}
	\label{Fig:Pillow3}
\end{figure}

\begin{example}
Consider $X_1=X_2=\mathbb{T}^2$ and suppose that $\overline{\omega}_1=pd\theta+qd\phi$ where $p,q$ are relatively prime integers and $\overline{\omega}_2=a\overline{\omega}_1$ where $a$ is an irrational number. Assume that the disks are $D_1$ and $D_2$ are such that their interiors intersect every leaf of the singular foliations $\tilde{\mathcal{F}}_{\omega_1}$ and
$\tilde{\mathcal{F}}_{\omega_2}$ at least twice. If one fixes the small disks in advance then the latter condition is automatically satisfied by taking the numbers $p$ and $q$ large enough. Under construction (a) one produces a transitive closed 1-form $\overline{\omega}$ such that all the leaves of the singular foliation $\tilde{\mathcal{F}}_{\omega}$ are non-compact, compare \cite[p. 138]{F}. 
\end{example}

\end{document}